\documentclass[a4paper,10pt]{amsart}
\usepackage{amsmath,amssymb,hyperref,multirow}

\theoremstyle{plain}
\newtheorem{theorem}{\bf Theorem}[section]
\newtheorem{proposition}[theorem]{\bf Proposition}
\newtheorem{lemma}[theorem]{\bf Lemma}
\newtheorem{corollary}[theorem]{\bf Corollary}

\theoremstyle{definition}
\newtheorem{example}[theorem]{\bf Example}

\newtheorem{remark}[theorem]{\bf Remark}

\hypersetup{hypertexnames=false}
\numberwithin{equation}{section}
\begin{document}

\title{On orders in quadratic number fields with unusual sets of distances}

\author{Andreas Reinhart}
\address{Institut f\"ur Mathematik und wissenschaftliches Rechnen, Karl-Franzens-Universit\"at Graz, NAWI Graz, Heinrichstra{\ss}e 36, 8010 Graz, Austria}
\email{andreas.reinhart@uni-graz.at}



\thanks{This work was supported by the Austrian Science Fund FWF, Project Number P36852-N}

\begin{abstract}
Let $\mathcal{O}$ be an order in an algebraic number field and suppose that the set of distances $\Delta(\mathcal{O})$ of $\mathcal{O}$ is nonempty (equivalently, $\mathcal{O}$ is not half-factorial). If $\mathcal{O}$ is seminormal (in particular, if $\mathcal{O}$ is a principal order), then $\min\Delta(\mathcal{O})=1$. So far, only a few examples of orders were found with $\min\Delta(\mathcal{O})>1$. We say that $\Delta(\mathcal{O})$ is unusual if $\min\Delta(\mathcal{O})>1$. In the present paper, we establish algebraic characterizations of orders $\mathcal{O}$ in real quadratic number fields with $\min\Delta(\mathcal{O})>1$. We also provide a classification of the real quadratic number fields that possess an order whose set of distances is unusual. As a consequence thereof, we revisit certain squarefree integers (cf. OEIS A135735) that were studied by A. J. Stephens and H. C. Williams.
\end{abstract}

\maketitle

\section{Introduction and terminology}\label{1}

Let $D$ be an integral domain with quotient field $K$ and let $a\in D$ be nonzero. A nonzero nonunit of $D$ is called an {\it atom} of $D$ if it cannot be written as a product of two nonunits of $D$. We say that $D$ is {\it atomic} if every nonzero nonunit of $D$ is a finite product of atoms of $D$. If $a=\eta\prod_{i=1}^k u_i\in D$, with $\eta$ a unit of $D$, $k$ a nonnegative integer and atoms $u_i$ of $D$, then $k$ is called a {\it factorization length} of $a$ and the set $\mathsf{L}(a)$ of all such $k$ is called the set of lengths of $a$. Note that if $D$ is Noetherian, then $D$ is atomic and $\mathsf{L}(b)$ is finite for each nonzero $b\in D$. Moreover, $D$ is said to be {\it half-factorial} if $\mathsf{L}(b)$ is a singleton for every nonzero $b\in D$. Clearly, every half-factorial domain is atomic. Besides that, $D$ is said to be {\it seminormal} if for each $x\in K$ with $x^2,x^3\in D$, it follows that $x\in D$. Let $\Delta(\mathsf{L}(a))$ be the set of distances of $\mathsf{L}(a)$, that is, the set of all positive integers $n$ for which we can find $k,\ell\in\mathsf{L}(a)$ with $k<\ell$ and $n=\ell-k$ such that $k$ and $\ell$ are the only lengths of $a$ that are between $k$ and $\ell$. The set of distances $\Delta(D)$ of $D$ is the union of all sets of distances over all sets of lengths of nonzero elements of $D$.

\smallskip
The set of distances is among the best investigated invariants in factorization theory. By definition, $\Delta(D)=\emptyset$ if and only if $D$ is half-factorial. If $\Delta(D)\not=\emptyset$, then $\min\Delta(D)=\gcd\Delta(D)$. There is a Dedekind domain $R_1$ such that $\Delta(R_1)$ is the set of positive integers and for every finite set $\Delta$ consisting of positive integers with $\min\Delta=\gcd\Delta$, there is a Dedekind domain $R_2$ with $\Delta(R_2)=\Delta$ (see \cite{Ge-Sc17a}).

\smallskip
The situation is quite different for orders in algebraic number fields. Let $\mathcal{O}$ be an order in an algebraic number field. Then its set of distances $\Delta(\mathcal{O})$ is finite (see \cite[Theorem 3.7.1]{Ge-HK06a}). Suppose that $\mathcal{O}$ is not half-factorial (for a recent characterization of half-factoriality see \cite{Ra24a}). If $\mathcal{O}$ is a principal order, then $\Delta(\mathcal{O})$ is an interval (see \cite{Ge-Yu12b,Ge-Zh19a}). If $\mathcal{O}$ is seminormal (this includes principal orders), then $\min\Delta(\mathcal{O})=1$ (see \cite[Theorem 1.1]{Ge-Zh16a}). So far, only a few examples of orders were found with $\min\Delta(\mathcal{O})>1$.

\smallskip
Let $\mathcal{O}$ be an order in a quadratic number field. In the present paper, we put our focus on the situation that $\min\Delta(\mathcal{O})>1$. Since sets of distances with this property are rather special, we say that $\Delta(\mathcal{O})$ is {\it unusual} if $\min\Delta(\mathcal{O})>1$. For the sake of simplicity, we set $\min\Delta(\mathcal{O})=0$ if $\Delta(\mathcal{O})=\emptyset$ (i.e., if $\mathcal{O}$ is half-factorial). The orders in quadratic number fields whose set of distances is unusual have already been characterized in \cite[Theorem 4.14]{Br-Ge-Re20}. In particular, it was shown in the aforementioned paper that if $\mathcal{O}$ is an order in a quadratic number field, then $\min\Delta(\mathcal{O})\in\{0,1,2\}$ and if $\mathcal{O}$ is an order in an imaginary quadratic number field, then $\min\Delta(\mathcal{O})\in\{0,1\}$. In particular, one can restrict to real quadratic number fields if one wants to describe the orders in quadratic number fields whose set of distances is unusual.

\smallskip
While the characterization in \cite{Br-Ge-Re20} provided a first important step, there are still many questions and open problems concerning orders in quadratic number fields whose set of distances is unusual. For instance, it is not known if the class number of the principal order of an order with unusual set of distances has to be $2$ or what the conductors of orders with unusual set of distances can look like. Another open problem is to find a simple criterion that provides all the orders in quadratic number fields with unusual set of distances.

\smallskip
Our paper is structured as follows. First, we present various necessary conditions for an order to have an unusual set of distances. It is shown that the class number (of the principal order) has to be $2$ (Theorem~\ref{theorem 2.6}) and that the order cannot be transfer Krull (Proposition~\ref{proposition 2.1}). Based on these results, we establish three new characterizations of orders with unusual set of distances (Corollaries~\ref{corollary 2.7} and~\ref{corollary 2.8}, Theorem~\ref{theorem 2.9}). We continue our study, by investigating the set of conductors of all orders in a given real quadratic number field whose set of distances is unusual (Propositions~\ref{proposition 3.1} and~\ref{proposition 3.4}). We present a new criterion to construct orders with unusual set of distances (Proposition~\ref{proposition 3.6}) and show (Example~\ref{example 3.7}) that this criterion is sharper than the criterion presented in \cite[Proposition 4.19]{Br-Ge-Re20}. As a consequence, we establish a simplified characterization for orders $\mathcal{O}$ with $\min\Delta(\mathcal{O})>1$ in the case that the norm of the fundamental unit is $-1$ (Theorem~\ref{theorem 3.9}). Besides showing that the aforementioned new criterion already characterizes the orders with unusual set of distances (Theorem~\ref{theorem 4.4}), we also give a classification of all real quadratic number fields that possess such an order (Theorem~\ref{theorem 5.4}). Finally, we present and discuss a variety of examples. To construct some of these examples, we need to revisit certain squarefree integers $d\in\mathbb{N}_{\geq 2}$ (cf. OEIS A135735) that were studied in \cite{Ch-Sa14,Mo13,St-Wi88}.

\bigskip
Next we discuss the used terminology. We denote by $\mathbb{P}$, $\mathbb{N}$, $\mathbb{N}_0$, $\mathbb{Z}$, $\mathbb{Q}$ the sets of prime numbers, positive integers, nonnegative integers, integers and rational numbers, respectively. For $a,b,k\in\mathbb{Z}$, let $[a,b]=\{x\in\mathbb{Z}:a\leq x\leq b\}$ and $\mathbb{N}_{\geq k}=\{x\in\mathbb{N}:x\geq k\}$. For each set $X$, let $|X|$ be the cardinality of $X$.

\smallskip
Let $I$ be an ideal of $D$, let $P$ be a prime ideal of $D$ and let $R$ be an overring of $D$ (i.e., an intermediate ring of $D$ and $K$). By $D^{\times}$, we denote the unit group of $D$ and by $\mathcal{A}(D)$ we denote the set of atoms of $D$. For subgroups $G,H$ of $K^{\times}$ with $H\subseteq G$ let $(G:H)=|G/H|$. Furthermore, let ${\rm spec}(D)$ denote the set of prime ideals of $D$. Recall that $I$ is {\it principal} if $I=aD=\{ax:x\in D\}$ for some $a\in D$. Moreover, $I$ is said to be {\it invertible} if there are some ideal $J$ of $D$ and some nonzero $y\in D$ such that $IJ=yD$. Let $\mathcal{I}^*(D)$ denote the monoid of invertible ideals of $D$ (equipped with ideal multiplication). By ${\rm Pic}(D)$ we denote the Picard group of $D$. It measures how far invertible ideals of $D$ are from being principal. Observe that ${\rm Pic}(D)$ is trivial if and only if each invertible ideal of $D$ is principal. The precise definition of the Picard group can be found in \cite{Ge-HK06a}.

\smallskip
For each invertible ideal $J$ of $D$, let $[J]\in {\rm Pic}(D)$ denote the class of $J$. Furthermore, let $\sqrt{I}=\{x\in D:x^k\in I$ for some $k\in\mathbb{N}\}=\bigcap_{P\in {\rm spec}(D),I\subseteq P} P$ be the radical of $I$. Moreover, let $X_P=\{x/y:x\in X,y\in D\setminus P\}$ for each subset $X$ of $K$. Observe that $I_P=ID_P$ and $R_P=RD_P$. Finally, we say that $D$ is {\it transfer Krull} if there exists a transfer homomorphism from the monoid of nonzero elements of $D$ into a Krull monoid. For the precise definitions of transfer homomorphisms and Krull monoids, we refer to \cite{Ge-HK06a}. Note that if $D$ is a Krull domain (i.e., the monoid of nonzero elements of $D$ is a Krull monoid) or $D$ is a half-factorial domain, then $D$ is transfer Krull (see e.g. \cite{Ba-Re22}).

\smallskip
Let $d\in\mathbb{N}_{\geq 2}$ be squarefree, let $K=\mathbb{Q}(\sqrt{d})$ and let $\mathcal{O}_K$ be the principal order in $K$ (i.e., the ring of algebraic integers in $K$). Moreover, set $\omega=\begin{cases}\sqrt{d}&\textnormal{if }d\equiv 2,3\mod 4,\\\frac{1+\sqrt{d}}{2}&\textnormal{if }d\equiv 1\mod 4,\end{cases}$ $\mathsf{d}_K=\begin{cases}
4d&\textnormal{if }d\equiv 2,3\mod 4,\\ d&\textnormal{if }d\equiv 1\mod 4,\end{cases}$ and let ${\rm N}:K\rightarrow\mathbb{Q}$ defined by ${\rm N}(a+b\sqrt{d})=a^2-db^2$ for each $a,b\in\mathbb{Q}$ be the norm map on $K$. It is well known that $\mathcal{O}_K=\mathbb{Z}[\omega]=\mathbb{Z}\oplus\omega\mathbb{Z}$.

\smallskip
We say that a subring $\mathcal{O}$ of $K$ with quotient field $K$ is an order in $K$ if $\mathcal{O}$ is a finitely generated $\mathbb{Z}$-module. It is well known that $\mathcal{O}_K$ is the largest order in $K$ with respect to inclusion. Let $f\in\mathbb{N}$. We set $\mathcal{O}_f=\mathbb{Z}+f\omega\mathbb{Z}$. Note that $\mathcal{O}_f$ is the unique order in $K$ with conductor $f$ (i.e., $\{x\in K:x\mathcal{O}_K\subseteq\mathcal{O}_f\}=f\mathcal{O}_K$) and $p\mathbb{Z}+f\omega\mathbb{Z}$ is the unique maximal ideal of $\mathcal{O}_f$ that contains $p\mathbb{Z}$ for each $p\in\mathbb{P}$ with $p\mid f$. For each nonzero ideal $J$ of $\mathcal{O}_f$, let ${\rm N}(J)=|\mathcal{O}_f/J|$, called the norm of $J$. Observe that ${\rm N}(a\mathcal{O}_f)=|{\rm N}(a)|$ for each nonzero $a\in\mathcal{O}_f$ (where $|\cdot|$ is the absolute value). Let $\varepsilon_K$ be the fundamental unit of $\mathcal{O}_K$ such that $\varepsilon_K>1$. If it is clear what field $K$ is considered, then we will write $\varepsilon$ instead of $\varepsilon_K$. If $\varepsilon=a+b\omega$ for $a,b\in\mathbb{N}$, then we say that $a$ is the {\it first component} of $\varepsilon$ and $b$ is the {\it second component} of $\varepsilon$.

\smallskip
For $a,b\in\mathbb{Z}$, let $\pmb{\Big(}\frac{a}{b}\pmb{\Big)}\in\{-1,0,1\}$ be the Kronecker symbol of $a$ modulo $b$. If $p\in\mathbb{P}$, then we say that $p$ is an inert prime, a ramified prime, a split prime (in $\mathcal{O}_K$) if $\pmb{\Big(}\frac{\mathsf{d}_K}{p}\pmb{\Big)}=-1$, $\pmb{\Big(}\frac{\mathsf{d}_K}{p}\pmb{\Big)}=0$, $\pmb{\Big(}\frac{\mathsf{d}_K}{p}\pmb{\Big)}=1$, respectively. We will use without further mention the fact that $2$ is ramified if and only if $d\not\equiv 1\mod 4$ and that $2$ is inert if and only if $d\equiv 5\mod 8$. Also note that for each odd $p\in\mathbb{P}$, $\pmb{\Big(}\frac{2}{p}\pmb{\Big)}=1$ if and only if $p\equiv 1,7\mod 8$ and $\pmb{\Big(}\frac{-1}{p}\pmb{\Big)}=1$ if and only if $p\equiv 1\mod 4$.

\smallskip
Let $p\in\mathbb{P}$ be such that $p\mid f$. Let $\mathsf{v}_p(f)$ be the $p$-adic exponent of $f$. Furthermore, let $\mathcal{I}_p^*(\mathcal{O}_f)=\{I\in\mathcal{I}^*(\mathcal{O}_f):\sqrt{I}\cap\mathbb{Z}\supseteq p\mathbb{Z}\}$ be the monoid of all invertible $p\mathbb{Z}+f\omega\mathbb{Z}$-primary ideals of $\mathcal{O}_f$ (including $\mathcal{O}_f$). Let $\mathcal{A}(\mathcal{I}^*(\mathcal{O}_f))=\{I\in\mathcal{I}^*(\mathcal{O}_f)\setminus\{\mathcal{O}_f\}:I$ is not a product of two proper ideals of $\mathcal{O}_f\}$ be the set of atoms of $\mathcal{I}^*(\mathcal{O}_f)$. Let $\mathcal{A}(\mathcal{I}_p^*(\mathcal{O}_f))=\mathcal{A}(\mathcal{I}^*(\mathcal{O}_f))\cap\mathcal{I}_p^*(\mathcal{O}_f)$.

We will use the fact that the class numbers $|{\rm Pic}(\mathcal{O}_f)|$ and $|{\rm Pic}(\mathcal{O}_K)|$ are connected by the formula

\[
|{\rm Pic}(\mathcal{O}_f)|=|{\rm Pic}(\mathcal{O}_K)|\frac{f}{(\mathcal{O}_K^{\times}:\mathcal{O}_f^{\times})}\prod_{p\in\mathbb{P},p\mid f}\left(1-\pmb{\Big(}\frac{\mathsf{d}_K}{p}\pmb{\Big)}\frac{1}{p}\right)
\]
without further mention. For a proof of the aforementioned equation we refer to \cite[Theorem 5.9.7.4]{HK13a}.

\medskip
\textit{Throughout this paper, let $d\in\mathbb{N}_{\geq 2}$ be squarefree. If not otherwise stated, then $K=\mathbb{Q}(\sqrt{d})$, $\mathcal{O}_K$, $\mathsf{d}_K$, $\varepsilon=\varepsilon_K$, $\omega$ and $\mathcal{O}_f$ for each $f\in\mathbb{N}$ are defined with respect to this fixed $d$.}

\section{Necessary conditions and characterizations}\label{2}

The main purpose of this section is to derive various necessary conditions for an order $\mathcal{O}$ (in a real quadratic number field) to satisfy $\min\Delta(\mathcal{O})>1$. Clearly, such an order cannot be half-factorial and the next result shows that it cannot even be transfer Krull. Furthermore, we present three new characterizations of orders $\mathcal{O}$ with $\min\Delta(\mathcal{O})>1$. Let $f\in\mathbb{N}$ be such that $\min\Delta(\mathcal{O}_f)>1$ and let $g\in\mathbb{N}$. As a byproduct of our investigations, we determine when $\Delta(\mathcal{O}_g)$ is unusual, where $g$ is either a divisor of $f$ or a specific multiple of $f$.

\begin{proposition}\label{proposition 2.1}
Let $f\in\mathbb{N}$ and let $\min\Delta(\mathcal{O}_f)>1$. Then $\mathcal{O}_f$ is not transfer Krull.
\end{proposition}

\begin{proof}
Let $\psi:{\rm spec}(\mathcal{O}_K)\rightarrow {\rm spec}(\mathcal{O}_f)$ be defined by $\psi(Q)=Q\cap\mathcal{O}_f$ for each $Q\in {\rm spec}(\mathcal{O}_K)$. It follows from \cite[Theorem 4.14]{Br-Ge-Re20} that $f$ is not divisible by a split prime, and hence $\psi$ is bijective. Therefore, $\mathcal{O}_f\subseteq\mathcal{O}_K$ is a root extension (i.e., for each $x\in\mathcal{O}_K$, there is some $k\in\mathbb{N}$ with $x^k\in\mathcal{O}_f$) by \cite[Corollary 3.7.2]{Ge-HK06a}. By \cite[Theorem 4.14]{Br-Ge-Re20} and \cite[Theorem 2.6.7]{Ge-HK06a}, we have $|{\rm Pic}(\mathcal{O}_K)|\leq 2$, and thus $\mathcal{O}_K$ is half-factorial by \cite[Theorem 1.7.3.6]{Ge-HK06a}. Since $\mathcal{O}_f$ is not half-factorial, it follows from \cite[Theorem 4.2]{Ba-Re22} that $\mathcal{O}_f$ is not transfer Krull.
\end{proof}

\begin{lemma}\label{lemma 2.2}
Let $t=|\{p\in\mathbb{P}:p\mid\mathsf{d}_K\}|$. Then $|{\rm Pic}(\mathcal{O}_K)|\geq 2^{t-2}$ and if ${\rm N}(\varepsilon)=-1$, then $|{\rm Pic}(\mathcal{O}_K)|\geq 2^{t-1}$.
\end{lemma}

\begin{proof}
This follows from \cite[Theorems 1.3.10, 6.1.7.1 and 6.5.2]{HK13a}.
\end{proof}

\begin{lemma}\label{lemma 2.3}
Let $p\in\mathbb{P}\setminus\{2\}$ be such that $p\nmid\mathsf{d}_K$. If $|{\rm Pic}(\mathcal{O}_p)|=|{\rm Pic}(\mathcal{O}_K)|$, then ${\rm N}(\varepsilon)=-1$ and $p\equiv 3\mod 4$.
\end{lemma}

\begin{proof}
This follows from \cite[Proposition 2.1 and Corollary 3.3]{Ch-Sa14}.
\end{proof}

\begin{lemma}\label{lemma 2.4}
Let $p\in\mathbb{P}$ be such that $p\equiv 1\mod 4$ and $d=2p$. Then $|{\rm Pic}(\mathcal{O}_K)|>1$.
\end{lemma}

\begin{proof}
This is an immediate consequence of \cite[Theorem 1]{Br74}.
\end{proof}

\begin{lemma}\label{lemma 2.5}
Let $f,g\in\mathbb{N}$ and $p\in\mathbb{P}$ be such that $p\mid g\mid f$ and ${\rm v}_p(g)={\rm v}_p(f)$. Then $\psi:\mathcal{I}^*_p(\mathcal{O}_f)\rightarrow\mathcal{I}^*_p(\mathcal{O}_g)$ defined by $\psi(A)=A\mathcal{O}_g$ for each $A\in\mathcal{I}^*_p(\mathcal{O}_f)$ is a monoid isomorphism.
\end{lemma}

\begin{proof}
Let $P$ be the unique maximal ideal of $\mathcal{O}_f$ with $P\cap\mathbb{Z}=p\mathbb{Z}$ and let $Q$ be the unique maximal ideal of $\mathcal{O}_g$ with $Q\cap\mathbb{Z}=p\mathbb{Z}$. Then $Q\cap\mathcal{O}_f=P$. If $A\in\mathcal{I}^*_p(\mathcal{O}_f)$, then $A$ is an invertible ideal of $\mathcal{O}_f$ with $\sqrt{A}\supseteq P$, and thus $A\mathcal{O}_g$ is an invertible ideal of $\mathcal{O}_g$ and $\sqrt{A\mathcal{O}_g}\supseteq Q$. Consequently, $\psi$ is a well-defined map. Since $\psi(\mathcal{O}_f)=\mathcal{O}_g$ and $\psi(IJ)=IJ\mathcal{O}_g=I\mathcal{O}_gJ\mathcal{O}_g=\psi(I)\psi(J)$ for all $I,J\in\mathcal{I}^*_p(\mathcal{O}_f)$, we see that $\psi$ is a monoid homomorphism. It follows from \cite[Proposition 3.3.1]{Br-Ge-Re20} that $(\mathcal{O}_f)_P=(\mathcal{O}_g)_Q$.

\smallskip
First we prove that $\psi$ is injective. Let $I,J\in\mathcal{I}^*_p(\mathcal{O}_f)$ be such that $\psi(I)=\psi(J)$. Clearly, $I$ is proper if and only if $J$ is proper. Therefore, we can assume without restriction that $I$ and $J$ are proper, and hence $I$ and $J$ are $P$-primary ideals of $\mathcal{O}_f$. Observe that $I_P=I(\mathcal{O}_f)_P=I(\mathcal{O}_g)_Q=(I\mathcal{O}_g)_Q=(J\mathcal{O}_g)_Q=J(\mathcal{O}_g)_Q=J(\mathcal{O}_f)_P=J_P$, and hence $I=I_P\cap\mathcal{O}_f=J_P\cap\mathcal{O}_f=J$.

\smallskip
Finally, we show that $\psi$ is surjective. Let $B\in\mathcal{I}^*_p(\mathcal{O}_g)$. Without restriction let $B$ be proper. Set $A=B\cap\mathcal{O}_f$. Then $\sqrt{A}=\sqrt{B}\cap\mathcal{O}_f=Q\cap\mathcal{O}_f=P$, and hence $\sqrt{A\mathcal{O}_g}=Q$. Therefore, $A$ is a $P$-primary ideal of $\mathcal{O}_f$ and $A\mathcal{O}_g$ is a $Q$-primary ideal of $\mathcal{O}_g$. We have $A_P=(B\cap\mathcal{O}_f)_P=B_P\cap (\mathcal{O}_f)_P=B(\mathcal{O}_f)_P\cap (\mathcal{O}_f)_P=B(\mathcal{O}_g)_Q\cap (\mathcal{O}_g)_Q=B_Q$, and hence $A_P\subseteq (A\mathcal{O}_g)_Q\subseteq B_Q=A_P$. Since $A\mathcal{O}_g$ and $B$ are $Q$-primary, we infer that $B=B_Q\cap\mathcal{O}_g=(A\mathcal{O}_g)_Q\cap\mathcal{O}_g=A\mathcal{O}_g$. Note that $A_P=B_Q$ is a principal ideal of $(\mathcal{O}_g)_Q$ (since $B$ is an invertible ideal of $\mathcal{O}_g$), and hence $A_P$ is a principal ideal of $(\mathcal{O}_f)_P$. This implies that $A$ is an invertible ideal of $\mathcal{O}_f$ (since $\mathcal{O}_f$ is Noetherian and $P$ is the only maximal ideal of $\mathcal{O}_f$ that contains $A$). It follows that $A\in\mathcal{I}^*_p(\mathcal{O}_f)$ and $\psi(A)=A\mathcal{O}_g=B$.
\end{proof}

\begin{theorem}\label{theorem 2.6}
Let $f,h\in\mathbb{N}$ be such that $\min\Delta(\mathcal{O}_f)>1$ and $h\mid f$.
\begin{enumerate}
\item $|{\rm Pic}(\mathcal{O}_K)|=2$.
\item $|\{p\in\mathbb{P}:p\mid f, p$ is inert$\}|\leq 2$, $|\{p\in\mathbb{P}:p\mid f, p$ is ramified$\}|\leq 3$ and $|\{p\in\mathbb{P}:p\mid f\}|\leq 4$.
\item If $h$ is not divisible by a ramified prime, then $\mathcal{O}_h$ is half-factorial. If $h$ is divisible by a ramified prime, then $\min\Delta(\mathcal{O}_h)>1$.
\end{enumerate}
\end{theorem}

\begin{proof}
It follows from \cite[Theorem 4.14]{Br-Ge-Re20} that $|{\rm Pic}(\mathcal{O}_f)|=2$, $f$ is squarefree, $f$ is not divisible by a split prime, $f$ is divisible by a ramified prime and for each prime divisor $p$ of $f$ and for each $A\in\mathcal{A}(\mathcal{I}^*_p(\mathcal{O}_f))$, $A$ is principal if and only if ${\rm N}(A)=p^2$. Let $\mathcal{I}$ be the set of all inert prime divisors of $f$ and let $\mathcal{L}$ be the set of all ramified prime divisors of $f$. Let $g=\prod_{p\in\mathcal{I}} p$. Then $\mathcal{O}_g$ is half-factorial and $g\in\{1\}\cup\mathbb{P}\cup\{2p:p\in\mathbb{P}\setminus\{2\}\}$ by \cite[Proposition 4.15]{Br-Ge-Re20}, and hence $|{\rm Pic}(\mathcal{O}_g)|=|{\rm Pic}(\mathcal{O}_K)|$ by \cite[Theorem 6.2]{Ge-Ka-Re15a}. Since $|{\rm Pic}(\mathcal{O}_f)|=2$, we infer that $|{\rm Pic}(\mathcal{O}_K)|\leq 2$ by \cite[Theorem 2.6.7]{Ge-HK06a}.

\bigskip
(1) Assume that $|{\rm Pic}(\mathcal{O}_K)|=1$. Let $t=|\{p\in\mathbb{P}:p\mid\mathsf{d}_K\}|$. By Lemma~\ref{lemma 2.2}, we have $t=1$ or ($t=2$ and ${\rm N}(\varepsilon)=1$). Observe that

\[
2=|{\rm Pic}(\mathcal{O}_f)|=|{\rm Pic}(\mathcal{O}_K)|\frac{f}{(\mathcal{O}_K^{\times}:\mathcal{O}_f^{\times})}\prod_{p\in\mathbb{P},p\mid f}\left(1-\pmb{\Big(}\frac{\mathsf{d}_K}{p}\pmb{\Big)}\frac{1}{p}\right)=\frac{1}{(\mathcal{O}_K^{\times}:\mathcal{O}_f^{\times})}\prod_{p\in\mathcal{I}} (p+1)\prod_{p\in\mathcal{L}} p
\]

and

\[
1=|{\rm Pic}(\mathcal{O}_g)|=|{\rm Pic}(\mathcal{O}_K)|\frac{g}{(\mathcal{O}_K^{\times}:\mathcal{O}_g^{\times})}\prod_{p\in\mathbb{P},p\mid g}\left(1-\pmb{\Big(}\frac{\mathsf{d}_K}{p}\pmb{\Big)}\frac{1}{p}\right)=\frac{1}{(\mathcal{O}_K^{\times}:\mathcal{O}_g^{\times})}\prod_{p\in\mathcal{I}} (p+1).
\]

\smallskip
Therefore, $(\mathcal{O}_K^{\times}:\mathcal{O}_g^{\times})=\prod_{p\in\mathcal{I}} (p+1)$ and $2(\mathcal{O}_K^{\times}:\mathcal{O}_f^{\times})=\prod_{p\in\mathcal{I}} (p+1)\prod_{p\in\mathcal{L}} p=(\mathcal{O}_K^{\times}:\mathcal{O}_g^{\times})\prod_{p\in\mathcal{L}} p$. It follows that $2(\mathcal{O}_g^{\times}:\mathcal{O}_f^{\times})=\prod_{p\in\mathcal{L}} p$. This implies that $2\mid f$ and $2$ is ramified. In particular, $d\equiv 2,3\mod 4$ and $g\in\{1\}\cup\mathbb{P}\setminus\{2\}$. There are some $r,s\in\mathbb{N}$ with $\varepsilon=r+s\sqrt{d}$.

\smallskip
Assume that $d\equiv 3\mod 4$. Then $t=2$, $d\in\mathbb{P}$ and ${\rm N}(\varepsilon)=1$. We infer by Lemma~\ref{lemma 2.3} that $g=1$, and thus $f\in\{2,2d\}$ and $(\mathcal{O}_K^{\times}:\mathcal{O}_f^{\times})=\frac{f}{2}\in\{1,d\}$. Moreover, $r^2-ds^2=1$ and $rs$ is even. If $(\mathcal{O}_K^{\times}:\mathcal{O}_f^{\times})=1$, then $s$ is even. If $(\mathcal{O}_K^{\times}:\mathcal{O}_f^{\times})=d$, then $\varepsilon^d\in\mathcal{O}_f$, and hence $0\equiv\sum_{i=0,i\equiv 1\mod 2}^d\binom{d}{i}r^{d-i}s^id^{\frac{i-1}{2}}\equiv s^d\mod 2$ and so $s$ is even. In any case, $s$ is even and $r$ is odd. Since $d$ is prime, there is some $\alpha\in\{-1,1\}$ such that $d\mid r+\alpha$. We have $\frac{r-\alpha}{2}\frac{r+\alpha}{2d}=(\frac{s}{2})^2$, and $\frac{r-\alpha}{2}$ and $\frac{r+\alpha}{2d}$ are coprime positive integers. Consequently, there are some $x,y\in\mathbb{N}$ with $x^2=\frac{r-\alpha}{2}$ and $y^2=\frac{r+\alpha}{2d}$. Observe that $|x^2-dy^2|=1$. Furthermore, $x<r$, which contradicts the fact that $\varepsilon$ is a fundamental unit of $\mathcal{O}_K$.

\smallskip
This implies $d\equiv 2\mod 4$. In particular, $\{A\in\mathcal{A}(\mathcal{I}^*_2(\mathcal{O}_f)):{\rm N}(A)=8\}=\{8\mathbb{Z}+f\sqrt{d}\mathbb{Z},8\mathbb{Z}+(4+f\sqrt{d})\mathbb{Z}\}$ by \cite[Theorem 3.6]{Br-Ge-Re20}, and thus $8\mathbb{Z}+f\sqrt{d}\mathbb{Z}$ and $8\mathbb{Z}+(4+f\sqrt{d})\mathbb{Z}$ are not principal. Observe that $8\mathbb{Z}+f\sqrt{d}\mathbb{Z}$ is principal if and only if there are some $a,b\in\mathbb{Z}$ such that $|8a^2-(\frac{f}{2})^2\frac{d}{2}b^2|=1$. Moreover, $8\mathbb{Z}+(4+f\sqrt{d})\mathbb{Z}$ is principal if and only if there are some $a,b\in\mathbb{Z}$ such that $|2(2a+b)^2-(\frac{f}{2})^2\frac{d}{2}b^2|=1$.

\smallskip
First let $t=1$. Then $d=2$, $f=2g$ and $\varepsilon=1+\sqrt{2}$. Since $1+\sqrt{2}\not\in\mathcal{O}_2$, it follows that $|{\rm Pic}(\mathcal{O}_2)|=|{\rm Pic}(\mathcal{O}_K)|$, and thus $g\not=1$. Therefore, $g\in\mathbb{P}\setminus\{2\}$. Since $g$ is inert, we have $\pmb{\Big(}\frac{2}{g}\pmb{\Big)}=-1$, and hence $g\equiv 3,5\mod 8$. We infer by Lemma~\ref{lemma 2.3} that $g\equiv 3\mod 4$. This implies that $g\equiv 3\mod 8$. There are some $u,v\in\mathbb{N}$ such that $(1+\sqrt{2})^{\frac{g+1}{2}}=u+v\sqrt{2}$. Observe that $u^2-2v^2={\rm N}(u+v\sqrt{2})={\rm N}((1+\sqrt{2})^{\frac{g+1}{2}})=(-1)^{\frac{g+1}{2}}=1$. Since $\pmb{\Big(}\frac{2}{g}\pmb{\Big)}=-1$, it follows that $\pmb{\Big(}\frac{2^{\frac{g-1}{2}}}{g}\pmb{\Big)}=(-1)^{\frac{g-1}{2}}=-1$, and thus $2^{\frac{g-1}{2}}\equiv -1\mod g$. There are some $u^{\prime},v^{\prime}\in\mathbb{N}$ with $(1+\sqrt{2})^g=u^{\prime}+v^{\prime}\sqrt{2}$. Note that $u^{\prime}\equiv 1\mod g$ and $v^{\prime}\equiv 2^{\frac{g-1}{2}}\equiv -1\mod g$. Therefore, $u^2+2v^2=u^{\prime}+2v^{\prime}\equiv -1\mod g$ and $2uv=u^{\prime}+v^{\prime}\equiv 0\mod g$. Consequently, $g\mid u$ or $g\mid v$. If $g\mid v$, then $u^2\equiv -1\mod g$, which contradicts the fact that $g\equiv 3\mod 8$. We infer that $g\mid u$. Furthermore, $v=\sum_{i=0,i\equiv 1\mod 2}^{\frac{g+1}{2}}\binom{\frac{g+1}{2}}{i}2^{\frac{i-1}{2}}\equiv\frac{g+1}{2}\equiv 0\mod 2$. Set $a=\frac{v}{2}$ and $b=\frac{u}{g}$. Then $a,b\in\mathbb{Z}$ and $|8a^2-(\frac{f}{2})^2\frac{d}{2}b^2|=1$. This implies that $8\mathbb{Z}+f\sqrt{d}\mathbb{Z}$ is principal, a contradiction.

\smallskip
Now let $t=2$ and ${\rm N}(\varepsilon)=1$. Then $d=2p$ for some $p\in\mathbb{P}\setminus\{2\}$. Note that $p\equiv 3\mod 4$ by Lemma~\ref{lemma 2.4}. Moreover, $g=1$ by Lemma~\ref{lemma 2.3}, and hence $f\in\{2,2p\}$. There are some $u,v\in\mathbb{N}$ such that $\varepsilon^p=u+v\sqrt{d}$. Observe that $r^2-2ps^2=1$, and thus $r$ is odd and $s$ is even. Assume that $4p\mid r+\alpha$ for some $\alpha\in\{-1,1\}$. Since $\frac{r-\alpha}{2}$ and $\frac{r+\alpha}{4p}$ are coprime positive integers and $\frac{r-\alpha}{2}\frac{r+\alpha}{4p}=(\frac{s}{2})^2$, there are some $x,y\in\mathbb{N}$ such that $x^2=\frac{r-\alpha}{2}$ and $y^2=\frac{r+\alpha}{4p}$. We have $|x^2-2py^2|=1$. Since $x<r$, this contradicts the fact that $\varepsilon$ is a fundamental unit of $\mathcal{O}_K$. Since $p$ is prime, there is some $\eta\in\{-1,1\}$ such that $p\mid r-\eta$. Note that $4\nmid r-\eta$ and $4\mid r+\eta$. Since $\frac{r+\eta}{4}$ and $\frac{r-\eta}{2p}$ are coprime positive integers and $\frac{r+\eta}{4}\frac{r-\eta}{2p}=(\frac{s}{2})^2$, there are some $x,y\in\mathbb{N}$ such that $x^2=\frac{r+\eta}{4}$ and $y^2=\frac{r-\eta}{2p}$. We infer that $2x^2-py^2=\eta$, and hence $\pmb{\Big(}\frac{2}{p}\pmb{\Big)}=\pmb{\Big(}\frac{\eta}{p}\pmb{\Big)}$.

\smallskip
Note that $u=\sum_{i=0,i\equiv 0\mod 2}^p\binom{p}{i}r^{p-i}s^i(2p)^{\frac{i}{2}}\equiv r^p+\frac{p(p-1)}{2}r^{p-2}s^22p\equiv r^p\equiv (-\eta)^p\equiv-\eta\mod 4$, and hence $4\mid u+\eta$. Furthermore, $u=\sum_{i=0,i\equiv 0\mod 2}^p\binom{p}{i}r^{p-i}s^i(2p)^{\frac{i}{2}}\equiv r^p+\frac{p(p-1)}{2}r^{p-2}s^22p\equiv r^{p-2}(r^2-p^2s^2)\mod p^3$. Set $z=r-\eta$. Then $r^2-p^2s^2=1+(2-p)\frac{r^2-1}{2}=1+(2-p)\frac{z(z+2\eta)}{2}\equiv 1+\eta (2-p)z+z^2\mod p^3$ and $r^{p-2}=\sum_{i=0}^{p-2}\binom{p-2}{i}\eta^{p-2-i}z^i\equiv\eta+(p-2)z+\eta\frac{(p-2)(p-3)}{2}z^2\equiv\eta+(p-2)z+3\eta z^2\mod p^3$. This implies that $u\equiv (\eta+(p-2)z+3\eta z^2)(1+\eta (2-p)z+z^2)\equiv\eta+(p-2)z+3\eta z^2+(2-p)z-\eta(p-2)^2z^2+\eta z^2=\eta-\eta (p^2-4p)z^2\equiv\eta\mod p^3$. We infer that $p^3\mid u-\eta$. Note that $u^2-dv^2={\rm N}(\varepsilon^p)=1$. Since $\frac{u+\eta}{4}$ and $\frac{u-\eta}{2p^3}$ are coprime positive integers and $\frac{u+\eta}{4}\frac{u-\eta}{2p^3}=(\frac{v}{2p})^2$, there are some $m,n\in\mathbb{Z}$ such that $\frac{u+\eta}{4}=m^2$ and $\frac{u-\eta}{2p^3}=n^2$. Consequently, $2m^2-p^3n^2=\eta$.

\smallskip
First let $p\equiv 3\mod 8$. Then $\pmb{\Big(}\frac{\eta}{p}\pmb{\Big)}=\pmb{\Big(}\frac{2}{p}\pmb{\Big)}=-1$, and thus $\eta=-1$ and $2m^2-p^3n^2=-1$. Clearly, $n$ is odd and if $m$ is even, then $-1\equiv -p^3n^2\equiv 5\mod 8$, a contradiction. We deduce that $m$ is odd. This implies that there are some $a,b\in\mathbb{Z}$ with $|2(2a+b)^2-(\frac{f}{2})^2\frac{d}{2}b^2|=1$, and hence $8\mathbb{Z}+(4+f\sqrt{d})\mathbb{Z}$ is principal, a contradiction.

\smallskip
Next let $p\equiv 7\mod 8$. Then $\pmb{\Big(}\frac{\eta}{p}\pmb{\Big)}=\pmb{\Big(}\frac{2}{p}\pmb{\Big)}=1$. Note that $\eta=1$ and $2m^2-p^3n^2=1$. Obviously, $n$ is odd and if $m$ is odd, then $1\equiv 2m^2-p^3n^2\equiv 3\mod 8$, a contradiction. We infer that $m$ is even. It follows that there are some $a,b\in\mathbb{Z}$ such that $|8a^2-(\frac{f}{2})^2\frac{d}{2}b^2|=1$, and thus $8\mathbb{Z}+f\sqrt{d}\mathbb{Z}$ is principal, a contradiction.

\bigskip
(2) We have already seen that $|\{p\in\mathbb{P}:p\mid f, p\textnormal{ is inert}\}|\leq 2$. Since $|{\rm Pic}(\mathcal{O}_K)|=2$ (by (1)), it follows from Lemma~\ref{lemma 2.2} that $|\{p\in\mathbb{P}:p\mid f, p\textnormal{ is ramified}\}|\leq |\{p\in\mathbb{P}:p\mid\mathsf{d}_K\}|\leq 3$. We obtain $|\{p\in\mathbb{P}:p\mid f\}|\leq 5$. Assume that $|\{p\in\mathbb{P}:p\mid f\}|=5$. Then $|\{p\in\mathbb{P}:p\mid f, p\textnormal{ is inert}\}|=2$ and $|\{p\in\mathbb{P}:p\mid\mathsf{d}_K\}|=3$. In particular, there exists some odd inert $q\in\mathbb{P}$ with $q\mid f$ and we have ${\rm N}(\varepsilon)=1$ by Lemma~\ref{lemma 2.2}. Since $q\mid g$ and $\mathcal{O}_g$ is half-factorial, we infer that $\mathcal{O}_q$ is half-factorial by \cite[Theorem 3.7.15]{Ge-HK06a}, and hence $|{\rm Pic}(\mathcal{O}_q)|=|{\rm Pic}(\mathcal{O}_K)|$ by \cite[Theorem 6.2.1]{Ge-Ka-Re15a}. Therefore, ${\rm N}(\varepsilon)=-1$ by Lemma~\ref{lemma 2.3}, a contradiction.

\bigskip
(3) First let $h$ be not divisible by a ramified prime. Then $h\mid g$, and hence $\mathcal{O}_g\subseteq\mathcal{O}_h$. Since $\mathcal{O}_g$ is half-factorial, we deduce by \cite[Theorem 3.7.15]{Ge-HK06a} that $\mathcal{O}_h$ is half-factorial. From now on let $h$ be divisible by a ramified prime.

\smallskip
Next we prove that $\varphi:{\rm Pic}(\mathcal{O}_f)\rightarrow {\rm Pic}(\mathcal{O}_h)$ defined by $\varphi([I])=[I\mathcal{O}_h]$ for each invertible ideal $I$ of $\mathcal{O}_f$ is a group isomorphism. It is straightforward to prove (and a simple consequence of the fact that $\mathcal{O}_f\subseteq\mathcal{O}_h$ are integral domains) that $\varphi$ is a well-defined group homomorphism. Now we show that $\varphi$ is surjective. Let $x\in {\rm Pic}(\mathcal{O}_h)\setminus\{[\mathcal{O}_h]\}$. By \cite[Corollary 2.11.16]{Ge-HK06a}, it follows that $x=[Q]$ for some nonprincipal invertible maximal ideal $Q$ of $\mathcal{O}_h$ with $f\not\in Q$. Set $P=Q\cap\mathcal{O}_f$. Note that $P$ is a maximal ideal of $\mathcal{O}_f$ and $f\not\in P$. Therefore, $P$ is an invertible ideal of $\mathcal{O}_f$ by \cite[Theorem 2.6.5 and Proposition 2.10.5]{Ge-HK06a}, and thus $(\mathcal{O}_f)_P=(\mathcal{O}_h)_Q=(\mathcal{O}_h)_{Q^{\prime}}$ is a Noetherian valuation domain for each maximal ideal $Q^{\prime}$ of $\mathcal{O}_h$ with $Q^{\prime}\cap\mathcal{O}_f=P$. Therefore, $Q$ is the only maximal ideal of $\mathcal{O}_h$ with $Q\cap\mathcal{O}_f=P$, and hence $P\mathcal{O}_h$ is $Q$-primary. Moreover, $P_P=(P\mathcal{O}_h)_Q=Q_Q$. This implies that $P\mathcal{O}_h=Q$, and hence $\varphi([P])=x$. It follows that $\varphi$ is surjective. We infer by (1) and \cite[Theorem 2.6.7]{Ge-HK06a} that $|{\rm Pic}(\mathcal{O}_f)|\geq |{\rm Pic}(\mathcal{O}_h)|\geq |{\rm Pic}(\mathcal{O}_K)|=2=|{\rm Pic}(\mathcal{O}_f)|$, and thus $|{\rm Pic}(\mathcal{O}_h)|=|{\rm Pic}(\mathcal{O}_f)|=|{\rm Pic}(\mathcal{O}_K)|=2$ and $\varphi$ is a group isomorphism.

\smallskip
Finally, we show that for each $r\in\mathbb{P}$ with $r\mid h$ and for each $A\in\mathcal{A}(\mathcal{I}^*_r(\mathcal{O}_h))$, $A$ is principal if and only if ${\rm N}(A)=r^2$. Let $r\in\mathbb{P}$ be such that $r\mid h$ and let $A\in\mathcal{A}(\mathcal{I}^*_r(\mathcal{O}_h))$. It follows from Lemma~\ref{lemma 2.5} that there is some $A^{\prime}\in\mathcal{A}(\mathcal{I}^*_r(\mathcal{O}_f))$ with $A^{\prime}\mathcal{O}_h=A$. Note that ${\rm N}(A)={\rm N}(A\mathcal{O}_K)={\rm N}(A^{\prime}\mathcal{O}_K)={\rm N}(A^{\prime})$ by \cite[pp. 99 and 100]{Ge-HK-Ka95}. Since $\varphi$ is a group isomorphism, $A$ is principal if and only if $A^{\prime}$ is principal if and only if ${\rm N}(A^{\prime})=r^2$ if and only if ${\rm N}(A)=r^2$. Now it follows from \cite[Theorem 4.14]{Br-Ge-Re20} that $\min\Delta(\mathcal{O}_h)>1$.
\end{proof}

\begin{corollary}\label{corollary 2.7}
Let $f\in\mathbb{N}$. Then $\min\Delta(\mathcal{O}_f)>1$ if and only if the following conditions are satisfied:
\begin{enumerate}
\item[(a)] $|{\rm Pic}(\mathcal{O}_f)|=|{\rm Pic}(\mathcal{O}_K)|=2$.
\item[(b)] $f$ is squarefree, $f$ is divisible by a ramified prime and $f$ is not divisible by a split prime.
\item[(c)] For each odd ramified $p\in\mathbb{P}$ with $p\mid f$, there exists some $A\in\mathcal{A}(\mathcal{I}^*_p(\mathcal{O}_f))$ such that ${\rm N}(A)=p^3$ and $A$ is not principal.
\item[(d)] If $2$ is ramified and $2\mid f$, then each $A\in\mathcal{A}(\mathcal{I}^*_2(\mathcal{O}_f))$ with ${\rm N}(A)=8$ is not principal.
\end{enumerate}
\end{corollary}

\begin{proof}
If $\min\Delta(\mathcal{O}_f)>1$, then the conditions (a)--(d) are satisfied by Theorem~\ref{theorem 2.6}(1) and \cite[Theorems 3.6 and 4.14]{Br-Ge-Re20}. Now let the conditions (a)--(d) be satisfied. By \cite[Theorem 4.14]{Br-Ge-Re20}, it remains to show that for each $p\in\mathbb{P}$ with $p\mid f$ and each $A\in\mathcal{A}(\mathcal{I}^*_p(\mathcal{O}_f))$, $A$ is principal if and only if ${\rm N}(A)=p^2$. Let $p\in\mathbb{P}$ be such that $p\mid f$ and let $A\in\mathcal{A}(\mathcal{I}^*_p(\mathcal{O}_f))$. If $p$ is inert, then ${\rm N}(A)=p^2$ by \cite[Theorem 3.6]{Br-Ge-Re20} and it follows along the same lines as in the proof of \cite[Proposition 4.19]{Br-Ge-Re20} that $A$ is principal. Now let $p$ be ramified. Note that ${\rm N}(A)\in\{p^2,p^3\}$ by \cite[Theorem 3.6]{Br-Ge-Re20}. If ${\rm N}(A)=p^3$, then it follows from conditions (c) and (d) and from \cite[Remark 4.16]{Br-Ge-Re20} that $A$ is not principal. Moreover, if ${\rm N}(A)=p^2$, then $A$ is principal by the claim in the proof of \cite[Theorem 4.14]{Br-Ge-Re20}.
\end{proof}

\begin{corollary}\label{corollary 2.8}
Let $f\in\mathbb{N}$ and let $\beta\in\{0,1\}$ be such that $f\mathsf{d}_K\equiv\beta\mod 2$. Then $\min\Delta(\mathcal{O}_f)>1$ if and only if the following conditions are satisfied:
\begin{enumerate}
\item[(a)] $|{\rm Pic}(\mathcal{O}_f)|=|{\rm Pic}(\mathcal{O}_K)|=2$.
\item[(b)] $f$ is squarefree, $f$ is divisible by a ramified prime and $f$ is not divisible by a split prime.
\item[(c)] For each odd ramified $p\in\mathbb{P}$ with $p\mid f$ and each $a,b\in\mathbb{Z}$, we have $|p(2pa+b\beta)^2-(\frac{f}{p})^2\frac{\mathsf{d}_K}{p}b^2|\not=4$.
\item[(d)] If $2$ is ramified and $2\mid f$, then for each $a,b\in\mathbb{Z}$, $|2a^2-(\frac{f}{2})^2\frac{\mathsf{d}_K}{2}b^2|\not=4$.
\end{enumerate}
\end{corollary}

\begin{proof}
For $p\in\mathbb{P}$ set $A_p=p^3\mathbb{Z}+\frac{p^2\beta+f\sqrt{\mathsf{d}_K}}{2}\mathbb{Z}$. Note that if $p$ is an odd ramified prime with ${\rm v}_p(f)=1$, then $A_p\in\mathcal{A}(\mathcal{I}^*_p(\mathcal{O}_f))$ and ${\rm N}(A_p)=p^3$ by \cite[Theorem 3.6]{Br-Ge-Re20} and $A_p$ is principal if and only if there are some $a,b\in\mathbb{Z}$ such that $|p(2pa+b\beta)^2-(\frac{f}{p})^2\frac{\mathsf{d}_K}{p}b^2|=4$. For the rest of this paragraph let $2$ be ramified and let ${\rm v}_2(f)=1$. Observe that $\{A\in\mathcal{A}(\mathcal{I}^*_2(\mathcal{O}_f)):{\rm N}(A)=8\}=\{8\mathbb{Z}+(2k+f\sqrt{d})\mathbb{Z}:k\in [0,3],k\equiv d\mod 2\}$ by \cite[Theorem 3.6]{Br-Ge-Re20}. Furthermore, $\{A\in\mathcal{A}(\mathcal{I}^*_2(\mathcal{O}_f)):{\rm N}(A)=8\}$ contains a principal ideal of $\mathcal{O}_f$ if and only if there are some $a,b\in\mathbb{Z}$ and some $k\in [0,3]$ such that $k\equiv d\mod 2$ and $|(8a+2kb)^2-f^2db^2|=8$ if and only if there are some $a,b\in\mathbb{Z}$ and some $k\in [0,3]$ such that $k\equiv d\mod 2$ and $|2(4a+kb)^2-(\frac{f}{2})^2\frac{\mathsf{d}_K}{2}b^2|=4$. Also note that if $a,b\in\mathbb{Z}$ are such that $|2a^2-(\frac{f}{2})^2\frac{\mathsf{d}_K}{2}b^2|=4$, then $a\equiv d\mod 2$, $b$ is odd and if $k\in [0,3]$ with $ab\equiv k\mod 4$, then $k\equiv d\mod 2$ and $a\equiv kb\mod 4$. We conclude that $\{A\in\mathcal{A}(\mathcal{I}^*_2(\mathcal{O}_f)):{\rm N}(A)=8\}$ contains a principal ideal of $\mathcal{O}_f$ if and only if there are some $a,b\in\mathbb{Z}$ such that $|2a^2-(\frac{f}{2})^2\frac{\mathsf{d}_K}{2}b^2|=4$.

\smallskip
If the conditions (a)--(d) are satisfied, then it is an immediate consequence of Corollary~\ref{corollary 2.7} that\linebreak $\min\Delta(\mathcal{O}_f)>1$. Now let $\min\Delta(\mathcal{O}_f)>1$. Then conditions (a) and (b) are satisfied by Corollary~\ref{corollary 2.7}. If $p$ is an odd ramified prime with $p\mid f$, then $A_p$ is not principal by \cite[Theorem 4.14]{Br-Ge-Re20}, and thus $|p(2pa+b\beta)^2-(\frac{f}{p})^2\frac{\mathsf{d}_K}{p}b^2|\not=4$ for all $a,b\in\mathbb{Z}$. If $2$ is ramified and $2\mid f$, then $\{A\in\mathcal{A}(\mathcal{I}^*_2(\mathcal{O}_f)):{\rm N}(A)=8\}$ does not contain a principal ideal of $\mathcal{O}_f$ by \cite[Theorem 4.14]{Br-Ge-Re20}, and hence $|2a^2-(\frac{f}{2})^2\frac{\mathsf{d}_K}{2}b^2|\not=4$ for all $a,b\in\mathbb{Z}$.
\end{proof}

\begin{theorem}\label{theorem 2.9}
Let $f\in\mathbb{N}$. Then $\min\Delta(\mathcal{O}_f)>1$ if and only if the following conditions are satisfied:
\begin{enumerate}
\item[(a)] $|{\rm Pic}(\mathcal{O}_f)|=|{\rm Pic}(\mathcal{O}_K)|=2$.
\item[(b)] $f$ is squarefree, $f$ is divisible by a ramified prime and $f$ is not divisible by a split prime.
\item[(c)] For each ramified $p\in\mathbb{P}$ with $p\mid f$ and each $a,b\in\mathbb{Z}$, $|pa^2-\frac{\mathsf{d}_K}{p}b^2|\not=4$.
\end{enumerate}
\end{theorem}

\begin{proof}
First let conditions (a)--(c) be satisfied. Let $p\in\mathbb{P}$ be odd and ramified with $p\mid f$ and let $\beta\in\{0,1\}$ be such that $f\mathsf{d}_K\equiv\beta\mod 2$. Assume that there are some $a,b\in\mathbb{Z}$ such that $|p(2pa+b\beta)^2-(\frac{f}{p})^2\frac{\mathsf{d}_K}{p}b^2|=4$. Then there are some $g,h\in\mathbb{Z}$ such that $|pg^2-\frac{\mathsf{d}_K}{p}h^2|=4$, a contradiction. Consequently, for all $a,b\in\mathbb{Z}$, $|p(2pa+b\beta)^2-(\frac{f}{p})^2\frac{\mathsf{d}_K}{p}b^2|\not=4$. Now let $2$ be ramified and $2\mid f$. Then $\mathsf{d}_K=4d$. Assume that there are some $a,b\in\mathbb{Z}$ such that $|2a^2-(\frac{f}{2})^2\frac{\mathsf{d}_K}{2}b^2|=4$. Then there are some $g,h\in\mathbb{Z}$ such that $|2g^2-\frac{\mathsf{d}_K}{2}h^2|=4$, a contradiction. Therefore, $|2a^2-(\frac{f}{2})^2\frac{\mathsf{d}_K}{2}b^2|\not=4$ for all $a,b\in\mathbb{Z}$. It follows from Corollary~\ref{corollary 2.8} that $\min\Delta(\mathcal{O}_f)>1$.

\smallskip
Now let $\min\Delta(\mathcal{O}_f)>1$. Then conditions (a) and (b) are satisfied by Corollary~\ref{corollary 2.8}. Let $p\in\mathbb{P}$ be ramified with $p\mid f$. Let $\beta\in\{0,1\}$ be such that $p\mathsf{d}_K\equiv\beta\mod 2$. By Theorem~\ref{theorem 2.6}(3), we have $\min\Delta(\mathcal{O}_p)>1$. If $p=2$, then $|pa^2-\frac{\mathsf{d}_K}{p}b^2|\not=4$ for all $a,b\in\mathbb{Z}$ by Corollary~\ref{corollary 2.8} (with $f=2$). Now let $p$ be odd. It follows from Corollary~\ref{corollary 2.8} (with $f=p$) that for all $g,h\in\mathbb{Z}$, $|p(2pg+h\beta)^2-\frac{\mathsf{d}_K}{p}h^2|\not=4$. Assume that there are some $x,y\in\mathbb{Z}$ such that $|px^2-\frac{\mathsf{d}_K}{p}y^2|=4$.

\medskip
CASE 1: $d\not\equiv 1\mod 4$. We have $\mathsf{d}_K=4d$, $\beta=0$, $|(pa)^2-db^2|=p$ for some $a,b\in\mathbb{Z}$ and $|(p^2g)^2-dh^2|\not=p$ for all $g,h\in\mathbb{Z}$. There are some $r,s\in\mathbb{N}$ such that $\varepsilon=r+s\sqrt{d}$. Since $|{\rm Pic}(\mathcal{O}_p)|=|{\rm Pic}(\mathcal{O}_K)|$ by Corollary~\ref{corollary 2.8}, we have $(\mathcal{O}_K^{\times}:\mathcal{O}_p^{\times})=p$. Therefore, $\varepsilon\not\in\mathcal{O}_p$, and thus $p\nmid s$. Since $|(pa)^2-db^2|=p$, it follows that $p\nmid b$. We infer that $p\nmid bs\frac{d}{p}$, and thus there is some $k\in [1,p]$ such that $p\mid ar+kbs\frac{d}{p}$. There are some $u,v\in\mathbb{N}$ with $\varepsilon^k=u+v\sqrt{d}$. Observe that $u\equiv r^k\mod p$ and $v\equiv kr^{k-1}s\mod p$. Note that $p\mid r^{k-1}(ar+kbs\frac{d}{p})=ar^k+kr^{k-1}bs\frac{d}{p}$, and hence $p\mid au+bv\frac{d}{p}$ and $p^2\mid pau+bvd$. Consequently, $pau+bvd=p^2g$ for some $g\in\mathbb{Z}$. Set $h=pav+bu$. Then $h\in\mathbb{Z}$ and $p=|{\rm N}(pa+b\sqrt{d})|=|{\rm N}((pa+b\sqrt{d})\varepsilon^k)|=|{\rm N}(p^2g+h\sqrt{d})|=|(p^2g)^2-dh^2|$, a contradiction.

\medskip
CASE 2: $d\equiv 1\mod 4$. Observe that $\mathsf{d}_K=d$, $\beta=1$, $|(pa)^2-db^2|=4p$ for some $a,b\in\mathbb{Z}$ and $|(2p^2g+ph)^2-dh^2|\not=4p$ for all $g,h\in\mathbb{Z}$. There are some $r,s\in\mathbb{N}$ such that $r\equiv s\mod 2$ and $\varepsilon=\frac{r}{2}+\frac{s}{2}\sqrt{d}$. Since $|(pa)^2-db^2|=4p$, we have $a\equiv b\mod 2$ and $p\nmid b$. Moreover, since $|{\rm Pic}(\mathcal{O}_p)|=|{\rm Pic}(\mathcal{O}_K)|$ by Corollary~\ref{corollary 2.8}, we have $(\mathcal{O}_K^{\times}:\mathcal{O}_p^{\times})=p$. Consequently, $\varepsilon\not\in\mathcal{O}_p$, and thus $p\nmid s$. This implies that $p\nmid s(b\frac{d}{p}-pa)$, and hence there is some $k\in [1,p]$ such that $r(a-b)+ks(b\frac{d}{p}-pa)\equiv 0\mod p$. There are some $u,v\in\mathbb{N}$ such that $u\equiv v\mod 2$ and $\varepsilon^k=\frac{u}{2}+\frac{v}{2}\sqrt{d}$. Note that $2^{k-1}u\equiv r^k\mod p$ and $2^{k-1}v\equiv kr^{k-1}s\mod p$. It follows that $2^{k-1}(u(a-b)+v(b\frac{d}{p}-pa))\equiv r^k(a-b)+kr^{k-1}s(b\frac{d}{p}-pa)\equiv 0\mod p$. Therefore, $u(a-b)+v(b\frac{d}{p}-pa)\equiv 0\mod p$, and thus $pu(a-b)+v(bd-p^2a)\equiv 0\mod p^2$. Since $a\equiv b\mod 2$ and $pu\equiv u\equiv v\mod 2$, we infer that $pu(a-b)+v(bd-p^2a)\equiv pu(a-b)+v(b-a)=(a-b)(pu-v)\equiv 0\mod 4$, and hence $pu(a-b)+v(bd-p^2a)\equiv 0\mod 4p^2$. We have $pu(a-b)+v(bd-p^2a)=4p^2g$ for some $g\in\mathbb{Z}$. Set $h=\frac{bu+pav}{2}$. Observe that $bu\equiv -pav\mod 2$, and so $h\in\mathbb{Z}$. Finally, $4p=|{\rm N}(pa+b\sqrt{d})|=|{\rm N}((pa+b\sqrt{d})\varepsilon^k)|=|{\rm N}(\frac{pau+bvd}{2}+h\sqrt{d})|=|{\rm N}(2p^2g+ph+h\sqrt{d})|=|(2p^2g+ph)^2-dh^2|$, a contradiction.
\end{proof}

\begin{corollary}\label{corollary 2.10}
Let $f,g,h\in\mathbb{N}$ be squarefree such that $f$ is only divisible by inert primes, $g$ and $h$ are coprime and products of ramified primes, $\min\Delta(\mathcal{O}_{fg})>1$ and $\min\Delta(\mathcal{O}_{fh})>1$. Then $\min\Delta(\mathcal{O}_{fgh})>1$.
\end{corollary}

\begin{proof}
Clearly, $fg$ and $fh$ are squarefree and $|{\rm Pic}(\mathcal{O}_{fg})|=|{\rm Pic}(\mathcal{O}_{fh})|=|{\rm Pic}(\mathcal{O}_K)|=2$ by Theorem~\ref{theorem 2.9}. Set $f^{\prime}=\prod_{p\in\mathbb{P},p\mid f} (p+1)$. Consequently, $(\mathcal{O}_K^{\times}:\mathcal{O}_{fg}^{\times})=f^{\prime}g$ and $(\mathcal{O}_K^{\times}:\mathcal{O}_{fh}^{\times})=f^{\prime}h$, and hence $f^{\prime}g,f^{\prime}h\mid (\mathcal{O}_K^{\times}:\mathcal{O}_{fgh}^{\times})$. Furthermore, $(\mathcal{O}_K^{\times}:\mathcal{O}_{fgh}^{\times})\mid f^{\prime}gh$. Observe that $g,h\mid\frac{(\mathcal{O}_K^{\times}:\mathcal{O}_{fgh}^{\times})}{f^{\prime}}\mid gh$. Since $g$ and $h$ are coprime, we infer that $\frac{(\mathcal{O}_K^{\times}:\mathcal{O}_{fgh}^{\times})}{f^{\prime}}=gh$, and thus $(\mathcal{O}_K^{\times}:\mathcal{O}_{fgh}^{\times})=f^{\prime}gh$. Therefore, $|{\rm Pic}(\mathcal{O}_{fgh})|=|{\rm Pic}(\mathcal{O}_K)|=2$. Since $fg$ and $fh$ are squarefree and $g$ and $h$ are coprime, we infer that $fgh$ is squarefree. Moreover, it follows from Theorem~\ref{theorem 2.9} that $fgh$ is divisible by a ramified prime, $fgh$ is not divisible by a split prime and for each ramified $p\in\mathbb{P}$ with $p\mid fgh$ and for all $a,b\in\mathbb{Z}$, $|pa^2-\frac{\mathsf{d}_K}{p}b^2|\not=4$. It is an immediate consequence of Theorem~\ref{theorem 2.9} that $\min\Delta(\mathcal{O}_{fgh})>1$.
\end{proof}

\section{Refined results and a characterization in the ``norm minus one'' case}\label{3}

In this section, we study the structure of the set of conductors of all orders $\mathcal{O}$ with $\min\Delta(\mathcal{O})>1$ in a fixed real quadratic number field. Furthermore, we show that if ${\rm N}(\varepsilon)=-1$ (i.e., if the norm of the fundamental unit of $\mathcal{O}_K$ is minus one), then the orders $\mathcal{O}$ with $\min\Delta(\mathcal{O})>1$ have a particularly simple description.

\medskip
We start with two definitions. Let $d\in\mathbb{N}_{\geq 2}$ be squarefree.
\begin{itemize}
\item Let $D_d=\{f\in\mathbb{N}:\min\Delta(\mathcal{O}_f)>1\}$, called the {\it set of unusual conductors} of $d$.
\item Let $D_d^{\prime}=\{f\in D_d:f$ is not divisible by an odd inert prime$\}$, called the {\it reduced set of unusual conductors} of $d$.
\end{itemize}

\begin{proposition}\label{proposition 3.1}
Let $d\in\mathbb{N}_{\geq 2}$ be squarefree.
\begin{enumerate}
\item $D_d^{\prime}\subseteq\{f\in\mathbb{N}:f\mid {\rm lcm}(2,d)\}$ and $D_d^{\prime}$ is finite.
\item If ${\rm N}(\varepsilon)=1$, then $D_d=D_d^{\prime}$ and $D_d$ is finite.
\item $D_d\not=\emptyset$ if and only if $D_d^{\prime}\not=\emptyset$.
\end{enumerate}
\end{proposition}

\begin{proof}
(1) Let $f\in D_d^{\prime}$. Then $\min\Delta(\mathcal{O}_f)>1$ and $f$ is not divisible by an odd inert prime. It follows from Theorem~\ref{theorem 2.9} that $f$ is squarefree, $f$ is divisible by a ramified prime and $f$ is not divisible by a split prime. Let $g$ be the product of all inert prime divisors of $f$. We infer by \cite[Proposition 4.15]{Br-Ge-Re20} that $g\in\{1\}\cup\mathbb{P}\cup\{2p:p\in\mathbb{P}\setminus\{2\}\}$. Observe that $g\in\{1,2\}$. Clearly, $\frac{f}{g}\mid\mathsf{d}_K$ and if $g=2$, then $2$ is inert and $d\equiv 5\mod 8$. If $d\equiv 1\mod 4$, then $f=g\frac{f}{g}\mid 2\mathsf{d}_K={\rm lcm}(2,d)$. Now let $d\not\equiv 1\mod 4$. Observe that $g=1$ and $f\mid\mathsf{d}_K=4d$. Since $f$ is squarefree, we infer that $f\mid 2^{1-{\rm v}_2(d)}d={\rm lcm}(2,d)$. Consequently, $D_d^{\prime}\subseteq\{f\in\mathbb{N}:f\mid {\rm lcm}(2,d)\}$. It is clear that $D_d^{\prime}$ is finite, since ${\rm lcm}(2,d)$ has only finitely many integer divisors.

\bigskip
(2) Let ${\rm N}(\varepsilon)=1$. By (1), it remains to show that for each $f\in D_d$, $f$ is not divisible by an odd inert prime. Let $f\in D_d$. Assume that there is some odd inert $p\in\mathbb{P}$ such that $p\mid f$. Note that $\mathcal{O}_p$ is half-factorial by Theorem~\ref{theorem 2.6}(3), and hence $|{\rm Pic}(\mathcal{O}_p)|=|{\rm Pic}(\mathcal{O}_K)|$ by \cite[Theorem 6.2.1]{Ge-Ka-Re15a}. Consequently, ${\rm N}(\varepsilon)=-1$ by Lemma~\ref{lemma 2.3}, a contradiction.

\bigskip
(3) This follows from Theorem~\ref{theorem 2.6}(3).
\end{proof}

\begin{example}\label{example 3.2}
Let $f\in\mathbb{N}$.
\begin{enumerate}
\item If $p\in\mathbb{P}$ is such that $p\equiv 3\mod 4$, $2^p-1\in\mathbb{P}$ and $(f,\mathsf{d}_K)=(5(2^p-1),40)$, then $f$ is the product of an inert prime and a ramified prime and $\min\Delta(\mathcal{O}_f)=2$.
\item If $(f,\mathsf{d}_K)=(2190,365)$, then $f$ is the product of two distinct inert primes and two distinct ramified primes and $\min\Delta(\mathcal{O}_f)=2$.
\end{enumerate}
\end{example}

\begin{proof}
(1) Let $p\in\mathbb{P}$ be such that $p\equiv 3\mod 4$, $2^p-1\in\mathbb{P}$ and $(f,\mathsf{d}_K)=(5(2^p-1),40)$. Set $M=2^p-1$. Note that $M\equiv 7\mod 8$. This implies that $\pmb{\Big(}\frac{2}{M}\pmb{\Big)}=1$. Moreover, $M=4^{\frac{p-1}{2}}\cdot 2-1\equiv (-1)^{\frac{p-1}{2}}\cdot 2-1\equiv 2\mod 5$. Therefore, $\pmb{\Big(}\frac{5}{M}\pmb{\Big)}=\pmb{\Big(}\frac{M}{5}\pmb{\Big)}=\pmb{\Big(}\frac{2}{5}\pmb{\Big)}=-1$ by the quadratic reciprocity theorem. We infer that $\pmb{\Big(}\frac{\mathsf{d}_K}{M}\pmb{\Big)}=\pmb{\Big(}\frac{2}{M}\pmb{\Big)}^3\pmb{\Big(}\frac{5}{M}\pmb{\Big)}=-1$, and thus $M$ is inert. Clearly, $5$ is ramified. We find that $f$ is the product of an inert prime and a ramified prime. It follows from \cite[p. 22]{HK13a} that $|{\rm Pic}(\mathcal{O}_K)|=2$. Also note that $\varepsilon=3+\sqrt{10}$, ${\rm N}(\varepsilon)=-1$, $d=10$ and $f\prod_{p\in\mathbb{P},p\mid f}\left(1-\pmb{\Big(}\frac{\mathsf{d}_K}{p}\pmb{\Big)}\frac{1}{p}\right)=2^p\cdot 5$. It follows from \cite[Proposition 3.6]{Ch-Sa14} that $(\mathcal{O}_K^{\times}:\mathcal{O}_M^{\times})=2^p$. Observe that $\varepsilon\not\in\mathcal{O}_5$, and thus $(\mathcal{O}_K^{\times}:\mathcal{O}_5^{\times})=5$. This implies that $2^p\mid (\mathcal{O}_K^{\times}:\mathcal{O}_f^{\times})$ and $5\mid (\mathcal{O}_K^{\times}:\mathcal{O}_f^{\times})\mid 2^p\cdot 5$. Since $2^p$ and $5$ are coprime, we infer that $(\mathcal{O}_K^{\times}:\mathcal{O}_f^{\times})=2^p\cdot 5$, and hence $|{\rm Pic}(\mathcal{O}_f)|=|{\rm Pic}(\mathcal{O}_K)|\frac{f}{(\mathcal{O}_K^{\times}:\mathcal{O}_f^{\times})}\prod_{p\in\mathbb{P},p\mid f}\left(1-\pmb{\Big(}\frac{\mathsf{d}_K}{p}\pmb{\Big)}\frac{1}{p}\right)=2$. Obviously, $5\equiv 1\mod 4$ and $\pmb{\Big(}\frac{d/5}{5}\pmb{\Big)}=-1$. Consequently, $\min\Delta(\mathcal{O}_f)=2$ by \cite[Proposition 4.19]{Br-Ge-Re20}.

\bigskip
(2) Let $(f,\mathsf{d}_K)=(2190,365)$. Since $d=\mathsf{d}_K\equiv 5\mod 8$, we infer that $2$ is inert. Moreover, $\pmb{\Big(}\frac{\mathsf{d}_K}{3}\pmb{\Big)}=\pmb{\Big(}\frac{2}{3}\pmb{\Big)}=-1$, and thus $3$ is inert. Clearly, $5\mid 365$, $73\mid 365$ and $2190=2\cdot 3\cdot 5\cdot 73$. Therefore, $f$ is the product of two distinct inert primes and two distinct ramified primes. It follows from \cite[p. 23]{HK13a} that $|{\rm Pic}(\mathcal{O}_K)|=2$. Also note that $\varepsilon=\frac{19+\sqrt{365}}{2}$, $\mathcal{O}_f=\mathbb{Z}+1095\sqrt{365}\mathbb{Z}$ and $f\prod_{p\in\mathbb{P},p\mid f}\left(1-\pmb{\Big(}\frac{\mathsf{d}_K}{p}\pmb{\Big)}\frac{1}{p}\right)=4380$. Consequently, $(\mathcal{O}_K^{\times}:\mathcal{O}_f^{\times})$ is a divisor of $4380=2^2\cdot 3\cdot 5\cdot 73$. Moreover, $\varepsilon^{4380/2}-730\sqrt{365}\in\mathcal{O}_f$, $\varepsilon^{4380/3}-\frac{367+1095\sqrt{365}}{2}\in\mathcal{O}_f$, $\varepsilon^{4380/5}-219\sqrt{365}\in\mathcal{O}_f$ and $\varepsilon^{4380/73}-810\sqrt{365}\in\mathcal{O}_f$. Therefore, $\varepsilon^{4380/2},\varepsilon^{4380/3},\varepsilon^{4380/5},\varepsilon^{4380/73}\not\in\mathcal{O}_f$, and hence $(\mathcal{O}_K^{\times}:\mathcal{O}_f^{\times})=4380$. We infer that $|{\rm Pic}(\mathcal{O}_f)|=|{\rm Pic}(\mathcal{O}_K)|\frac{f}{(\mathcal{O}_K^{\times}:\mathcal{O}_f^{\times})}\prod_{p\in\mathbb{P},p\mid f}\left(1-\pmb{\Big(}\frac{\mathsf{d}_K}{p}\pmb{\Big)}\frac{1}{p}\right)=2$. Since $5\equiv 1\mod 4$ and $\pmb{\Big(}\frac{d/5}{5}\pmb{\Big)}=\pmb{\Big(}\frac{73}{5}\pmb{\Big)}=-1$, it follows from \cite[Proposition 4.19]{Br-Ge-Re20} that $\min\Delta(\mathcal{O}_f)=2$.
\end{proof}

\begin{lemma}\label{lemma 3.3}
Let $q,r,s\in\mathbb{P}$ be odd and distinct such that $d=qrs\equiv 1\mod 4$ and ${\rm N}(\varepsilon)=1$. Then there are some $p\in\{q,r,s\}$ and some $a,b\in\mathbb{Z}$ such that $|pa^2-\frac{\mathsf{d}_K}{p}b^2|=4$ and if $q\equiv r\equiv s\equiv 1\mod 4$, then $|{\rm Pic}(\mathcal{O}_K)|>2$.
\end{lemma}

\begin{proof}
Clearly, $\mathsf{d}_K=d$ and there are some $u^{\prime},v^{\prime}\in\mathbb{N}$ such that $u^{\prime}\equiv v^{\prime}\mod 2$ and $\varepsilon=\frac{u^{\prime}}{2}+\frac{v^{\prime}}{2}\sqrt{d}$. First let $u^{\prime}$ and $v^{\prime}$ be odd. Note that ${u^{\prime}}^2-d{v^{\prime}}^2=4$ and $(u^{\prime}+2)(u^{\prime}-2)=d{v^{\prime}}^2$. Assume that there is some $e\in\{-1,1\}$ such that $d\mid u^{\prime}+2e$. Then $(u^{\prime}-2e)\frac{u^{\prime}+2e}{d}={v^{\prime}}^2$ and $u^{\prime}-2e$ and $\frac{u^{\prime}+2e}{d}$ are coprime odd positive integers. Consequently, there are some odd $a,b\in\mathbb{N}$ such that $a^2=u^{\prime}-2e$ and $b^2=\frac{u^{\prime}+2e}{d}$. We infer that $|a^2-db^2|=4$ and $a<u^{\prime}$, which contradicts the fact that $\varepsilon$ is a fundamental unit of $\mathcal{O}_K$. Consequently, $d\nmid u^{\prime}+2$ and $d\nmid u^{\prime}-2$. Observe that there are some $p\in\{q,r,s\}$ and some $e\in\{-1,1\}$ such that $p\mid u^{\prime}+2e$ and $\frac{d}{p}\mid u^{\prime}-2e$. It follows that $\frac{u^{\prime}+2e}{p}\frac{u^{\prime}-2e}{d/p}={v^{\prime}}^2$ and $\frac{u^{\prime}+2e}{p}$ and $\frac{u^{\prime}-2e}{d/p}$ are coprime positive integers. This implies that there are some $a,b\in\mathbb{Z}$ such that $a^2=\frac{u^{\prime}+2e}{p}$ and $b^2=\frac{u^{\prime}-2e}{d/p}$. Obviously, $|pa^2-\frac{\mathsf{d}_K}{p}b^2|=4$.

Now let $u^{\prime}$ and $v^{\prime}$ be even. Set $u=\frac{u^{\prime}}{2}$ and $v=\frac{v^{\prime}}{2}$. Then $u,v\in\mathbb{N}$, $\varepsilon=u+v\sqrt{d}$ and $u^2-dv^2=1$. Note that $(u+1)(u-1)=dv^2$, $u$ is odd and $v$ is even (since $d\equiv 1\mod 4$). Assume that there is some $e\in\{-1,1\}$ such that $d\mid u+e$. Then $\frac{u-e}{2}\frac{u+e}{2d}=(\frac{v}{2})^2$ and $\frac{u-e}{2}$ and $\frac{u+e}{2d}$ are coprime positive integers. This implies that there are some $a,b\in\mathbb{N}$ such that $a^2=\frac{u-e}{2}$ and $b^2=\frac{u+e}{2d}$. Therefore, $|a^2-db^2|=1$ and $a<u$, which contradicts the fact that $\varepsilon$ is a fundamental unit of $\mathcal{O}_K$. We have that $d\nmid u+1$ and $d\nmid u-1$. Consequently, there are some $p\in\{q,r,s\}$ and some $e\in\{-1,1\}$ such that $p\mid u+e$ and $\frac{d}{p}\mid u-e$. Observe that $\frac{u+e}{2p}\frac{u-e}{2d/p}=(\frac{v}{2})^2$ and $\frac{u+e}{2p}$ and $\frac{u-e}{2d/p}$ are coprime positive integers. We infer that there are some $g,h\in\mathbb{Z}$ such that $g^2=\frac{u+e}{2p}$ and $h^2=\frac{u-e}{2d/p}$. Moreover, $|p(2g)^2-\frac{\mathsf{d}_K}{p}(2h)^2|=4$, and hence there are some $a,b\in\mathbb{Z}$ such that $|pa^2-\frac{\mathsf{d}_K}{p}b^2|=4$.

Finally, let $q\equiv r\equiv s\equiv 1\mod 4$. Without restriction let $|qa^2-\frac{\mathsf{d}_K}{q}b^2|=4$ for some $a,b\in\mathbb{Z}$. Then $\pmb{\Big(}\frac{q}{r}\pmb{\Big)}=\pmb{\Big(}\frac{q}{s}\pmb{\Big)}=1$. It follows from \cite[Theorem 5(b)]{Br74} that $|{\rm Pic}(\mathcal{O}_K)|>2$.
\end{proof}

Next we provide a refinement of Theorem~\ref{theorem 2.6}(2). If $f\in\mathbb{N}$, then let $\mathcal{I}_f$ be the set of all inert prime divisors of $f$ and let $\mathcal{L}_f$ be the set of all ramified prime divisors of $f$.

\begin{proposition}\label{proposition 3.4}
Let $\Omega=\{(|\mathcal{I}_f|,|\mathcal{L}_f|):f,d\in\mathbb{N}, d>1, d$ is squarefree and $\min\Delta(\mathcal{O}_f)>1\}$. Then $\Omega=\{(0,1),(0,2),(0,3),(1,1),(1,2),(2,1),(2,2)\}$.
\end{proposition}

\begin{proof}
Note that $\Omega\subseteq\{(0,1),(0,2),(0,3),(1,1),(1,2),(2,1),(2,2),(1,3)\}$ by Theorems~\ref{theorem 2.6}(2) and~\ref{theorem 2.9}. Moreover, $\{(0,1),(0,2),(0,3),(1,1),(1,2),(2,1),(2,2)\}\subseteq\Omega$ by \cite[Example 4.22]{Br-Ge-Re20} and Example~\ref{example 3.2}(2). It remains to show that $(1,3)\not\in\Omega$. Assume that $(1,3)\in\Omega$. Then there are some $f,d\in\mathbb{N}$ such that $d>1$, $d$ is squarefree, $\min\Delta(\mathcal{O}_f)>1$ and $f$ is the product of an inert prime and three distinct ramified primes. Note that $|{\rm Pic}(\mathcal{O}_f)|=|{\rm Pic}(\mathcal{O}_K)|=2$ by Theorem~\ref{theorem 2.9}. It follows from Lemma~\ref{lemma 2.2} that $|\{p\in\mathbb{P}:p\mid\mathsf{d}_K\}|=|\{p\in\mathbb{P}:p\mid f,p$ is ramified$\}|=3$ and ${\rm N}(\varepsilon)=1$. We infer by Lemma~\ref{lemma 2.3} and Proposition~\ref{proposition 3.1}(2) that $\{p\in\mathbb{P}:p\mid f,p$ is inert$\}=\{2\}$, and hence $\mathsf{d}_K=d\equiv 5\mod 8$ and $f=2d$. Moreover, there are some odd distinct $q,r,s\in\mathbb{P}$ such that $d=qrs$. We infer by Theorem~\ref{theorem 2.9} that for each $p\in\{q,r,s\}$ and each $a,b\in\mathbb{Z}$, $|pa^2-\frac{\mathsf{d}_K}{p}b^2|\not=4$, which contradicts Lemma~\ref{lemma 3.3}.
\end{proof}

\begin{lemma}\label{lemma 3.5}
Let $f\in\mathbb{N}$ and let $p\in\mathbb{P}$ be ramified with ${\rm v}_p(f)=1$ and such that for each $\alpha\in\{-1,1\}$, $(p$ is odd and $\pmb{\Big(}\frac{\alpha d/p}{p}\pmb{\Big)}=-1)$ or there is some odd $q\in\mathbb{P}$ such that $q\mid df$ and $\pmb{\Big(}\frac{-\alpha p}{q}\pmb{\Big)}=-1$. Then for each $I\in\mathcal{A}(\mathcal{I}^*_p(\mathcal{O}_f))$ with ${\rm N}(I)=p^3$, $I$ is not principal.
\end{lemma}

\begin{proof}
Let $\beta\in\{0,1\}$ be such that $f\mathsf{d}_K\equiv\beta\mod 2$. If $p$ is odd, then $\{I\in\mathcal{A}(\mathcal{I}^*_p(\mathcal{O}_f)):{\rm N}(I)=p^3\}=\{p^3\mathbb{Z}+\left(p^2k+\frac{p^2\beta+f\sqrt{\mathsf{d}_K}}{2}\right)\mathbb{Z}:k\in [0,p-1]\}$ by \cite[Theorem 3.6]{Br-Ge-Re20}. If $p=2$, then $\{I\in\mathcal{A}(\mathcal{I}^*_p(\mathcal{O}_f)):{\rm N}(I)=p^3\}=\{8\mathbb{Z}+(2k+f\sqrt{d})\mathbb{Z}:k\in [0,3],k\equiv d\mod 2\}$ by \cite[Theorem 3.6]{Br-Ge-Re20}. Assume that there is some $I\in\mathcal{A}(\mathcal{I}^*_p(\mathcal{O}_f))$ such that ${\rm N}(I)=p^3$ and $I$ is principal.

\medskip
CASE 1: $p$ is odd. There is some $k\in [0,p-1]$ such that $I=p^3\mathbb{Z}+\left(p^2k+\frac{p^2\beta+f\sqrt{\mathsf{d}_K}}{2}\right)\mathbb{Z}$. Since $I$ is principal, there are some $a,b\in\mathbb{Z}$ with $I=\left(p^3a+p^2kb+\frac{p^2\beta+f\sqrt{\mathsf{d}_K}}{2}b\right)\mathcal{O}_f$, and hence $p^3={\rm N}(I)=\frac{1}{4}|p^4(2pa+2bk+b\beta)^2-f^2b^2\mathsf{d}_K|$. Consequently, there are some $\alpha\in\{-1,1\}$ and some $m,n\in\mathbb{Z}$ such that $p^4m^2-f^2n^2d=-4p^3\alpha$. It follows that $pm^2-\frac{f^2n^2}{p^2}\frac{d}{p}=-4\alpha$, and hence $\pmb{\Big(}\frac{-d/p}{p}\pmb{\Big)}=\pmb{\Big(}\frac{(f^2n^2/p^2)(-d/p)}{p}\pmb{\Big)}=\pmb{\Big(}\frac{-4\alpha}{p}\pmb{\Big)}=\pmb{\Big(}\frac{-\alpha}{p}\pmb{\Big)}$. We infer that $\pmb{\Big(}\frac{\alpha d/p}{p}\pmb{\Big)}=1$. This implies that there is some odd $q\in\mathbb{P}$ such that $q\mid df$ and $\pmb{\Big(}\frac{-\alpha p}{q}\pmb{\Big)}=-1$. Since $p^4m^2-f^2n^2d=-4p^3\alpha$, we have $\pmb{\Big(}\frac{-\alpha p}{q}\pmb{\Big)}=\pmb{\Big(}\frac{-4p^3\alpha}{q}\pmb{\Big)}=1$, a contradiction.

\medskip
CASE 2: $p=2$. There is some $k\in [0,3]$ such that $I=8\mathbb{Z}+(2k+f\sqrt{d})\mathbb{Z}$. Moreover, there are some odd $q,r\in\mathbb{P}$ such that $q,r\mid df$ and $\pmb{\Big(}\frac{2}{q}\pmb{\Big)}=\pmb{\Big(}\frac{-2}{r}\pmb{\Big)}=-1$. Since $I$ is principal, there are some $a,b\in\mathbb{Z}$ with $I=(8a+2kb+bf\sqrt{d})\mathcal{O}_f$, and thus $8={\rm N}(I)=|(8a+2bk)^2-b^2f^2d|$. We infer that there is some $\alpha\in\{-1,1\}$ such that $(8a+2bk)^2-b^2f^2d=8\alpha$. It follows that $\pmb{\Big(}\frac{2\alpha}{q}\pmb{\Big)}=\pmb{\Big(}\frac{2\alpha}{r}\pmb{\Big)}=1$, and hence $\pmb{\Big(}\frac{\alpha}{q}\pmb{\Big)}=\pmb{\Big(}\frac{2}{q}\pmb{\Big)}=-1$. Therefore, $\alpha=-1$, and so $\pmb{\Big(}\frac{-2}{r}\pmb{\Big)}=1$, a contradiction.
\end{proof}

\begin{proposition}\label{proposition 3.6}
Let $f\in\mathbb{N}$ be squarefree such that $f$ is divisible by a ramified prime, $f$ is not divisible by a split prime and $|{\rm Pic}(\mathcal{O}_f)|=|{\rm Pic}(\mathcal{O}_K)|=2$. If for each $\alpha\in\{-1,1\}$ and each ramified $p\in\mathbb{P}$ with $p\mid f$, $(p$ is odd and $\pmb{\Big(}\frac{\alpha d/p}{p}\pmb{\Big)}=-1)$ or $($there is some odd $q\in\mathbb{P}$ such that $q\mid df$ and $\pmb{\Big(}\frac{-\alpha p}{q}\pmb{\Big)}=-1)$, then $\min\Delta(\mathcal{O}_f)=2$.
\end{proposition}

\begin{proof}
This is an immediate consequence of Corollary~\ref{corollary 2.7}, Lemma~\ref{lemma 3.5} and \cite[Theorem 4.14]{Br-Ge-Re20}.
\end{proof}

Next we want to emphasize that Proposition~\ref{proposition 3.6} is sharper than \cite[Proposition 4.19]{Br-Ge-Re20}.

\begin{example}\label{example 3.7}
If $(f,\mathsf{d}_K)=(3,168)$, then for each $p\in\mathbb{P}$ with $p\mid df$, $p\not\equiv 1\mod 4$ and $\min\Delta(\mathcal{O}_f)=2$.
\end{example}

\begin{proof}
Let $(f,\mathsf{d}_K)=(3,168)$. The first statement is clearly true. Moreover, $d=42$, $f$ is a ramified prime, $|{\rm Pic}(\mathcal{O}_K)|=2$ by \cite[p. 22]{HK13a} and $\varepsilon=13+2\sqrt{42}$. Clearly, $\varepsilon\not\in\mathcal{O}_f$ and $f\prod_{p\in\mathbb{P},p\mid f}\left(1-\pmb{\Big(}\frac{\mathsf{d}_K}{p}\pmb{\Big)}\frac{1}{p}\right)=3$, and thus $(\mathcal{O}_K^{\times}:\mathcal{O}_f^{\times})=3$ and $|{\rm Pic}(\mathcal{O}_f)|=|{\rm Pic}(\mathcal{O}_K)|=2$. If $\alpha=1$, then $f$ is odd and $\pmb{\Big(}\frac{\alpha d/f}{f}\pmb{\Big)}=-1$. If $\alpha=-1$, then $7$ is an odd prime divisor of $df$ and $\pmb{\Big(}\frac{-\alpha f}{7}\pmb{\Big)}=-1$. Consequently, $\min\Delta(\mathcal{O}_f)=2$ by Proposition~\ref{proposition 3.6}.
\end{proof}

\begin{lemma}\label{lemma 3.8}
Let ${\rm N}(\varepsilon)=-1$ and suppose that one of the following conditions is satisfied:
\begin{enumerate}
\item[(a)] There is some $p\in\mathbb{P}$ with $p\equiv 1\mod 8$ and $d=2p$.
\item[(b)] There are some $p,q\in\mathbb{P}$ with $p\equiv q\equiv 1\mod 4$, $\pmb{\Big(}\frac{p}{q}\pmb{\Big)}=1$ and $d=pq$.
\end{enumerate}
Then $|{\rm Pic}(\mathcal{O}_K)|>2$.
\end{lemma}

\begin{proof}
This is an immediate consequence of \cite[Theorems 1 and 2]{Br74}.
\end{proof}

\begin{theorem}\label{theorem 3.9}
Let $f\in\mathbb{N}$ and let ${\rm N}(\varepsilon)=-1$. Then $\min\Delta(\mathcal{O}_f)>1$ if and only if the following conditions are satisfied:
\begin{enumerate}
\item[(a)] $|{\rm Pic}(\mathcal{O}_f)|=|{\rm Pic}(\mathcal{O}_K)|=2$.
\item[(b)] $f$ is squarefree, $f$ is divisible by a ramified prime and $f$ is not divisible by a split prime.
\end{enumerate}
\end{theorem}

\begin{proof}
By Theorem~\ref{theorem 2.9} and Proposition~\ref{proposition 3.6}, it remains to show that if conditions (a) and (b) are satisfied, then for each ramified $p\in\mathbb{P}$ with $p\mid f$ and each $\alpha\in\{-1,1\}$, it follows that $((p$ is odd and $\pmb{\Big(}\frac{\alpha d/p}{p}\pmb{\Big)}=-1)$ or there is some odd $q\in\mathbb{P}$ such that $q\mid df$ and $\pmb{\Big(}\frac{-\alpha p}{q}\pmb{\Big)}=-1)$. Let conditions (a) and (b) be satisfied, let $p\in\mathbb{P}$ be ramified with $p\mid f$ and let $\alpha\in\{-1,1\}$. Set $t=|\{r\in\mathbb{P}:r\mid\mathsf{d}_K\}|$. Since ${\rm N}(\varepsilon)=-1$, it follows from Lemma~\ref{lemma 2.2} that $t\in\{1,2\}$. If $t=1$, then $|{\rm Pic}(\mathcal{O}_K)|$ is odd (see e.g. \cite[p. 100]{Br74}), a contradiction. Therefore, there is some $q\in\mathbb{P}$ with $q\not=p$ and $\{r\in\mathbb{P}:r\mid\mathsf{d}_K\}=\{p,q\}$. Since ${\rm N}(\varepsilon)=-1$, we have $p,q\not\equiv 3\mod 4$. In particular, $\mathsf{d}_K\not=4r$ for all $r\in\mathbb{P}$.

\medskip
CASE 1: $p\equiv q\equiv 1\mod 4$. Then $\mathsf{d}_K=d=pq$ and since $|{\rm Pic}(\mathcal{O}_K)|=2$, we infer by Lemma~\ref{lemma 3.8} that $\pmb{\Big(}\frac{q}{p}\pmb{\Big)}=-1$. Consequently, $p$ is odd and $\pmb{\Big(}\frac{\alpha d/p}{p}\pmb{\Big)}=\pmb{\Big(}\frac{q}{p}\pmb{\Big)}=-1$.

\medskip
CASE 2: $p=2$. Then $q\equiv 1\mod 4$ and $d=2q$. We infer by Lemma~\ref{lemma 3.8} that $q\equiv 5\mod 8$. Clearly, $q$ is odd, $q\mid df$ and $\pmb{\Big(}\frac{-\alpha p}{q}\pmb{\Big)}=\pmb{\Big(}\frac{2}{q}\pmb{\Big)}=-1$.

\medskip
CASE 3: $q=2$. Then $p\equiv 1\mod 4$ and $d=2p$. It follows from Lemma~\ref{lemma 3.8} that $p\equiv 5\mod 8$. Moreover, $p$ is odd and $\pmb{\Big(}\frac{\alpha d/p}{p}\pmb{\Big)}=\pmb{\Big(}\frac{2}{p}\pmb{\Big)}=-1$.
\end{proof}

\begin{corollary}\label{corollary 3.10}
Let $|{\rm Pic}(\mathcal{O}_K)|=2$.
\begin{enumerate}
\item $D_d\not=\emptyset$ if and only if there is some ramified $p\in\mathbb{P}$ such that $|{\rm Pic}(\mathcal{O}_p)|=|{\rm Pic}(\mathcal{O}_K)|$ and for all $a,b\in\mathbb{Z}$, $|pa^2-\frac{\mathsf{d}_K}{p}b^2|\not=4$.
\item If ${\rm N}(\varepsilon)=-1$, then $D_d\not=\emptyset$ if and only if there is some ramified $p\in\mathbb{P}$ such that $|{\rm Pic}(\mathcal{O}_p)|=|{\rm Pic}(\mathcal{O}_K)|$.
\end{enumerate}
\end{corollary}

\begin{proof}
(1) If there is some ramified $p\in\mathbb{P}$ such that $|{\rm Pic}(\mathcal{O}_p)|=|{\rm Pic}(\mathcal{O}_K)|$ and for all $a,b\in\mathbb{Z}$, $|pa^2-\frac{\mathsf{d}_K}{p}b^2|\not=4$, then $p\in D_d$ by Theorem~\ref{theorem 2.9}. Now let $D_d\not=\emptyset$. Then there is some $f\in\mathbb{N}$ with $\min\Delta(\mathcal{O}_f)>1$. It follows from Theorem~\ref{theorem 2.9} that there is some ramified prime $p$ with $p\mid f$. We infer by Theorem~\ref{theorem 2.6}(3) that $\min\Delta(\mathcal{O}_p)>1$, and hence $|{\rm Pic}(\mathcal{O}_p)|=|{\rm Pic}(\mathcal{O}_K)|$ and for all $a,b\in\mathbb{Z}$, $|pa^2-\frac{\mathsf{d}_K}{p}b^2|\not=4$ by Theorem~\ref{theorem 2.9}.

\bigskip
(2) This easily follows from (1) and Theorem~\ref{theorem 3.9}.
\end{proof}

\section{Another characterization in the general case}\label{4}

In this section we prove that the sufficient criterion in Proposition~\ref{proposition 3.6} characterizes the orders $\mathcal{O}$ for which $\Delta(\mathcal{O})$ is unusual. This characterization is especially useful to count the orders whose set of distances is unusual and whose discriminant is at most $n$ for each $n\in\mathbb{N}$. As a first step, we show that the aforementioned criterion is equivalent to a criterion with slightly different conditions.

\begin{remark}\label{remark 4.1}
Let $f\in\mathbb{N}$ be squarefree such that $f$ is divisible by a ramified prime, $f$ is not divisible by a split prime and $|{\rm Pic}(\mathcal{O}_f)|=|{\rm Pic}(\mathcal{O}_K)|=2$. The following conditions are equivalent:
\begin{enumerate}
\item For each ramified $p\in\mathbb{P}$ with $p\mid f$ and each $\alpha\in\{-1,1\}$, $(p$ is odd and $\pmb{\Big(}\frac{\alpha d/p}{p}\pmb{\Big)}=-1)$ or there is some odd ramified $q\in\mathbb{P}$ such that $\pmb{\Big(}\frac{-\alpha p}{q}\pmb{\Big)}=-1$.
\item For each ramified $p\in\mathbb{P}$ with $p\mid f$ and each $\alpha\in\{-1,1\}$, $(p$ is odd and $\pmb{\Big(}\frac{\alpha d/p}{p}\pmb{\Big)}=-1)$ or there is some odd $q\in\mathbb{P}$ such that $q\mid df$ and $\pmb{\Big(}\frac{-\alpha p}{q}\pmb{\Big)}=-1$.
\end{enumerate}
\end{remark}

\begin{proof}
(1) $\Rightarrow$ (2) This is clear, since for every odd ramified $q\in\mathbb{P}$, we have that $q\mid df$.

\bigskip
(2) $\Rightarrow$ (1) If ${\rm N}(\varepsilon)=-1$, then a careful examination of the proof of Theorem~\ref{theorem 3.9} shows that condition (1) is satisfied. Now let ${\rm N}(\varepsilon)=1$. We infer by Lemma~\ref{lemma 2.3} and Theorem~\ref{theorem 2.6}(3) that $f$ is not divisible by an odd inert prime. Consequently, each odd $q\in\mathbb{P}$ with $q\mid df$ is ramified. Therefore, condition (1) is satisfied.
\end{proof}

\begin{lemma}\label{lemma 4.2}
Suppose that there are some distinct $p,q\in\mathbb{P}$ such that $p\equiv q\equiv 3\mod 4$ and $(d=p$ or $d=2p$ or $d=pq)$. Then $|{\rm Pic}(\mathcal{O}_K)|$ is odd.
\end{lemma}

\begin{proof}
This follows from \cite[Theorems 1 and 2]{Br74}.
\end{proof}

\begin{lemma}\label{lemma 4.3}
Suppose that one of the following conditions is satisfied:
\begin{enumerate}
\item[(a)] There are some distinct $p,q\in\mathbb{P}$ such that $p\equiv 1\mod 8$, $q\equiv 3\mod 4$, $\pmb{\Big(}\frac{p}{q}\pmb{\Big)}=1$ and $d=pq$.
\item[(b)] There are some odd distinct $p,q\in\mathbb{P}$ such that $\big(p\equiv q\equiv 1\mod 8$ or $p\equiv q\equiv 7\mod 8$ or $(p\equiv 1\mod 8$ or $q\equiv 1\mod 8)$ and $\pmb{\Big(}\frac{p}{q}\pmb{\Big)}=1\big)$ and $d=2pq$.
\item[(c)] There are some distinct $p,q,r\in\mathbb{P}$ such that $p\equiv 1\mod 4$, $q\equiv r\equiv 3\mod 4$, $\pmb{\Big(}\frac{p}{q}\pmb{\Big)}=\pmb{\Big(}\frac{p}{r}\pmb{\Big)}=1$ and $d=pqr$.
\end{enumerate}
Then $|{\rm Pic}(\mathcal{O}_K)|>2$.
\end{lemma}

\begin{proof}
This is an immediate consequence of \cite[Theorems 3--5]{Br74}.
\end{proof}

\begin{theorem}\label{theorem 4.4}
Let $f\in\mathbb{N}$. Then $\min\Delta(\mathcal{O}_f)>1$ if and only if the following conditions are satisfied:
\begin{enumerate}
\item[(a)] $|{\rm Pic}(\mathcal{O}_f)|=|{\rm Pic}(\mathcal{O}_K)|=2$.
\item[(b)] $f$ is squarefree, $f$ is divisible by a ramified prime and $f$ is not divisible by a split prime.
\item[(c)] For each ramified $p\in\mathbb{P}$ with $p\mid f$ and each $\alpha\in\{-1,1\}$, $(p$ is odd and $\pmb{\Big(}\frac{\alpha d/p}{p}\pmb{\Big)}=-1)$ or there is some odd ramified $q\in\mathbb{P}$ such that $\pmb{\Big(}\frac{-\alpha p}{q}\pmb{\Big)}=-1$.
\end{enumerate}
\end{theorem}

\begin{proof}
If conditions (a)--(c) are satisfied, then it is an immediate consequence of Proposition~\ref{proposition 3.6} and Remark~\ref{remark 4.1} that $\min\Delta(\mathcal{O}_f)>1$.

\smallskip
Now let $\min\Delta(\mathcal{O}_f)>1$. By Theorem~\ref{theorem 2.9}, conditions (a) and (b) are satisfied and for each ramified $p\in\mathbb{P}$ with $p\mid f$ and each $a,b\in\mathbb{Z}$, $|pa^2-\frac{\mathsf{d}_K}{p}b^2|\not=4$. By Theorem~\ref{theorem 3.9} and its proof, we can assume without restriction that ${\rm N}(\varepsilon)=1$. Let $p\in\mathbb{P}$ be ramified with $p\mid f$ and let $\alpha\in\{-1,1\}$. We have to show that ($p$ is odd and $\pmb{\Big(}\frac{\alpha d/p}{p}\pmb{\Big)}=-1$) or there is some odd ramified $\ell\in\mathbb{P}$ such that $\pmb{\Big(}\frac{-\alpha p}{\ell}\pmb{\Big)}=-1$. Set $t^{\prime}=|\{\ell\in\mathbb{P}:\ell\mid d\}|$. It follows from Lemma~\ref{lemma 2.2} that $t^{\prime}\in [1,3]$. If $t^{\prime}=1$, then $|{\rm Pic}(\mathcal{O}_K)|$ is odd by Lemma~\ref{lemma 4.2} and \cite[p. 100]{Br74}, a contradiction.

\smallskip
If $t^{\prime}=2$ and $d\equiv 2\mod 4$, then it follows from Lemma~\ref{lemma 4.2} that $d=2q$ for some $q\in\mathbb{P}$ with $q\equiv 1\mod 4$. If $t^{\prime}=2$ and $d\equiv 1\mod 4$, then we infer by Lemma~\ref{lemma 4.2} that $d=qr$ for some distinct $q,r\in\mathbb{P}$ with $q\equiv r\equiv 1\mod 4$. If $t^{\prime}=2$ and $d\equiv 3\mod 4$, then $d=qr$ for some $q,r\in\mathbb{P}$ with $q\equiv 1\mod 4$ and $r\equiv 3\mod 4$. If $t^{\prime}=3$ and $d\equiv 2\mod 4$, then $d=2qr$ for some odd distinct $q,r\in\mathbb{P}$. If $t^{\prime}=3$ and $d\equiv 1\mod 4$, then $d=qrs$ for some odd distinct $q,r,s\in\mathbb{P}$ with $q\equiv 1\mod 4$ and $r\equiv s\mod 4$. If $t^{\prime}=3$ and $d\equiv 3\mod 4$, then $\mathsf{d}_K$ has four distinct prime divisors, which contradicts Lemma~\ref{lemma 2.2}. In particular, we need to consider the following cases.

\medskip
CASE 1: $d=2q$ for some $q\in\mathbb{P}$ with $q\equiv 1\mod 4$. Then $\mathsf{d}_K=8q$, $p\in\{2,q\}$ and for all $a,b\in\mathbb{Z}$, $|pa^2-\frac{d}{p}b^2|\not=1$. There are some $u,v\in\mathbb{N}$ such that $\varepsilon=u+v\sqrt{d}$. Since $u^2-dv^2=1$, we infer that $u$ is odd and $v$ is even. Moreover, $(u+1)(u-1)=dv^2$. Assume that there is some $e\in\{-1,1\}$ such that $4q\mid u+e$. Then $\frac{u-e}{2}\frac{u+e}{4q}=(\frac{v}{2})^2$ and $\frac{u-e}{2}$ and $\frac{u+e}{4q}$ are coprime positive integers. This implies that there are some $a,b\in\mathbb{N}$ such that $a^2=\frac{u-e}{2}$ and $b^2=\frac{u+e}{4q}$, and hence $|a^2-db^2|=1$. Since $a<u$, this contradicts the fact that $\varepsilon$ is a fundamental unit of $\mathcal{O}_K$. Consequently, $4q\nmid u+1$ and $4q\nmid u-1$. Therefore, there is some $e\in\{-1,1\}$ such that $d\mid u+e$ and $4\mid u-e$. Again since $\frac{u-e}{4}\frac{u+e}{d}=(\frac{v}{2})^2$ and $\frac{u-e}{4}$ and $\frac{u+e}{d}$ are coprime positive integers, there are some $g,h\in\mathbb{Z}$ such that $g^2=\frac{u-e}{4}$ and $h^2=\frac{u+e}{d}$. Observe that $|2g^2-qh^2|=1$, and thus $|pa^2-\frac{d}{p}b^2|=1$ for some $a,b\in\mathbb{Z}$, a contradiction.

\medskip
CASE 2: $d=qr$ for some distinct $q,r\in\mathbb{P}$ with $q\equiv r\equiv 1\mod 4$. Then $\mathsf{d}_K=qr$, $p\in\{q,r\}$ and for all $a,b\in\mathbb{Z}$, $|qa^2-rb^2|\not=4$. There are some $u^{\prime},v^{\prime}\in\mathbb{N}$ such that $u^{\prime}\equiv v^{\prime}\mod 2$ and $\varepsilon=\frac{u^{\prime}}{2}+\frac{v^{\prime}}{2}\sqrt{d}$.

First we assume that $u^{\prime}$ and $v^{\prime}$ are odd. Observe that ${u^{\prime}}^2-d{v^{\prime}}^2=4$, and hence $(u^{\prime}+2)(u^{\prime}-2)=d{v^{\prime}}^2$. Assume that there is some $e\in\{-1,1\}$ such that $d\mid u^{\prime}+2e$. Then $(u^{\prime}-2e)\frac{u^{\prime}+2e}{d}={v^{\prime}}^2$ and $u^{\prime}-2e$ and $\frac{u^{\prime}+2e}{d}$ are coprime odd positive integers. Consequently, there are some odd $a,b\in\mathbb{N}$ such that $a^2=u^{\prime}-2e$ and $b^2=\frac{u^{\prime}+2e}{d}$. This implies that $|a^2-db^2|=4$ and $a<u^{\prime}$, which contradicts the fact that $\varepsilon$ is a fundamental unit of $\mathcal{O}_K$. Therefore, there is some $e\in\{-1,1\}$ such that $q\mid u^{\prime}+2e$ and $r\mid u^{\prime}-2e$. We infer that $\frac{u^{\prime}+2e}{q}\frac{u^{\prime}-2e}{r}={v^{\prime}}^2$ and $\frac{u^{\prime}+2e}{q}$ and $\frac{u^{\prime}-2e}{r}$ are coprime positive integers. It follows that there are some $a,b\in\mathbb{Z}$ such that $a^2=\frac{u^{\prime}+2e}{q}$ and $b^2=\frac{u^{\prime}-2e}{r}$. Observe that $|qa^2-rb^2|=4$, a contradiction.

Now let $u^{\prime}$ and $v^{\prime}$ be even. Set $u=\frac{u^{\prime}}{2}$ and $v=\frac{v^{\prime}}{2}$. Then $u^2-dv^2=1$ and $(u+1)(u-1)=dv^2$. Clearly, $u$ is odd and $v$ is even. Assume that $d\mid u+e$ for some $e\in\{-1,1\}$. Then $\frac{u-e}{2}\frac{u+e}{2d}=(\frac{v}{2})^2$ and $\frac{u-e}{2}$ and $\frac{u+e}{2d}$ are coprime positive integers. In particular, there are some $a,b\in\mathbb{N}$ such that $a^2=\frac{u-e}{2}$ and $b^2=\frac{u+e}{2d}$. It follows that $|a^2-db^2|=1$ and $a<u$, which contradicts the fact that $\varepsilon$ is a fundamental unit of $\mathcal{O}_K$. Consequently, there is some $e\in\{-1,1\}$ such that $q\mid u+e$ and $r\mid u-e$. We have that $\frac{u+e}{2q}\frac{u-e}{2r}=(\frac{v}{2})^2$ and $\frac{u+e}{2q}$ and $\frac{u-e}{2r}$ are coprime positive integers. We infer that $a^2=\frac{u+e}{2q}$ and $b^2=\frac{u-e}{2r}$ for some $a,b\in\mathbb{Z}$. Obviously, $|q(2a)^2-r(2b)^2|=4$, a contradiction.

\medskip
CASE 3: $d=qr$ for some distinct $q,r\in\mathbb{P}$ with $q\equiv 1\mod 4$ and $r\equiv 3\mod 4$. Then $\mathsf{d}_K=4qr$, $p\in\{2,q,r\}$ and ($q\equiv 5\mod 8$ or $\pmb{\Big(}\frac{q}{r}\pmb{\Big)}=-1$) by Lemma~\ref{lemma 4.3}. There are some $u,v\in\mathbb{N}$ such that $\varepsilon=u+v\sqrt{d}$. Clearly, $u^2-dv^2=1$, and hence ($u$ is odd and $v$ is even) or ($u$ is even and $v$ is odd). Moreover, $(u+1)(u-1)=dv^2$.

First let $u$ be odd and let $v$ be even. Assume that there is some $e\in\{-1,1\}$ such that $d\mid u+e$. Then $\frac{u-e}{2}\frac{u+e}{2d}=(\frac{v}{2})^2$ and $\frac{u-e}{2}$ and $\frac{u+e}{2d}$ are coprime positive integers. Therefore, there are some $a,b\in\mathbb{N}$ such that $a^2=\frac{u-e}{2}$ and $b^2=\frac{u+e}{2d}$. We have $|a^2-db^2|=1$ and $a<u$, which contradicts the fact that $\varepsilon$ is a fundamental unit of $\mathcal{O}_K$. Consequently, there is some $e\in\{-1,1\}$ such that $q\mid u+e$ and $r\mid u-e$. Observe that $\frac{u+e}{2q}\frac{u-e}{2r}=(\frac{v}{2})^2$ and $\frac{u+e}{2q}$ and $\frac{u-e}{2r}$ are coprime positive integers. This implies that $a^2=\frac{u+e}{2q}$ and $b^2=\frac{u-e}{2r}$ for some $a,b\in\mathbb{Z}$. Note that $|qa^2-rb^2|=1$, and thus $\pmb{\Big(}\frac{q}{r}\pmb{\Big)}=\pmb{\Big(}\frac{r}{q}\pmb{\Big)}=1$. Therefore, $q\equiv 5\mod 8$. If $p$ is odd, then for all $g,h\in\mathbb{Z}$, $|qg^2-rh^2|\not=1$ (since $|px^2-\frac{\mathsf{d}_K}{p}y^2|\not=4$ for all $x,y\in\mathbb{Z}$), a contradiction. This implies that $p=2$. Since $q\equiv 5\mod 8$, it follows that $\pmb{\Big(}\frac{-\alpha p}{q}\pmb{\Big)}=-1$.

Now let $u$ be even and let $v$ be odd. First let there be some $e\in\{-1,1\}$ such that $d\mid u+e$. Then $(u-e)\frac{u+e}{d}=v^2$ and $u-e$ and $\frac{u+e}{d}$ are coprime positive integers. Obviously, there are some $a,b\in\mathbb{Z}$ such that $a^2=u-e$ and $b^2=\frac{u+e}{d}$. We have $|a^2-db^2|=2$, and hence $\pmb{\Big(}\frac{2}{q}\pmb{\Big)}=1$ and $q\equiv 1\mod 8$. We infer that $\pmb{\Big(}\frac{r}{q}\pmb{\Big)}=\pmb{\Big(}\frac{q}{r}\pmb{\Big)}=-1$. If $p=2$, then $|g^2-dh^2|\not=2$ for all $g,h\in\mathbb{Z}$, a contradiction. This implies that $p$ is odd. If $p=q$, then $\pmb{\Big(}\frac{\alpha d/p}{p}\pmb{\Big)}=-1$. If $p=r$, then $\pmb{\Big(}\frac{-\alpha p}{q}\pmb{\Big)}=-1$. From now on we can assume that $d\nmid u+1$ and $d\nmid u-1$. Consequently, there is some $e\in\{-1,1\}$ such that $q\mid u+e$ and $r\mid u-e$. Since $\frac{u+e}{q}\frac{u-e}{r}=v^2$ and $\frac{u+e}{q}$ and $\frac{u-e}{r}$ are coprime positive integers, there are some $a,b\in\mathbb{Z}$ such that $a^2=\frac{u+e}{q}$ and $b^2=\frac{u-e}{r}$. We infer that $|qa^2-rb^2|=2$, and thus $\pmb{\Big(}\frac{q}{r}\pmb{\Big)}=\pmb{\Big(}\frac{r}{q}\pmb{\Big)}=\pmb{\Big(}\frac{2}{q}\pmb{\Big)}$. If $\pmb{\Big(}\frac{q}{r}\pmb{\Big)}=1$, then $\pmb{\Big(}\frac{2}{q}\pmb{\Big)}=1$ and $q\equiv 1\mod 8$, a contradiction. It follows that $\pmb{\Big(}\frac{q}{r}\pmb{\Big)}=\pmb{\Big(}\frac{r}{q}\pmb{\Big)}=\pmb{\Big(}\frac{2}{q}\pmb{\Big)}=-1$. If $p\in\{2,r\}$, then $\pmb{\Big(}\frac{-\alpha p}{q}\pmb{\Big)}=-1$. Finally, if $p=q$, then $\pmb{\Big(}\frac{\alpha d/p}{p}\pmb{\Big)}=-1$.

\medskip
CASE 4: $d=2qr$ for some odd distinct $q,r\in\mathbb{P}$. Then $\mathsf{d}_K=8qr$ and $p\in\{2,q,r\}$. Let $m,n\in\{1,3,5,7\}$ be such that $q\equiv m\mod 8$ and $r\equiv n\mod 8$. There are some $u,v\in\mathbb{N}$ such that $\varepsilon=u+v\sqrt{d}$. Note that $u^2-dv^2=1$, $(u+1)(u-1)=dv^2$, $u$ is odd and $v$ is even. Assume that $p=2$. Then it follows from Theorems~\ref{theorem 2.6}(3) and~\ref{theorem 2.9} that $|{\rm Pic}(\mathcal{O}_2)|=|{\rm Pic}(\mathcal{O}_K)|$. Consequently, $(\mathcal{O}_K^{\times}:\mathcal{O}_2^{\times})=2$, and thus $\varepsilon\not\in\mathcal{O}_2$. This contradicts the fact that $v$ is even. Therefore, $p\in\{q,r\}$.

\medskip
CLAIM: There are some $a,b\in\mathbb{Z}$ with $|a^2-db^2|=2$ or there are some $a,b\in\mathbb{Z}$ such that $|qa^2-2rb^2|=1$ or there are some $a,b\in\mathbb{Z}$ such that $|ra^2-2qb^2|=1$.

\medskip\smallskip
We prove the claim. Assume that there is some $e\in\{-1,1\}$ such that $2d\mid u+e$. Then $\frac{u-e}{2}\frac{u+e}{2d}=(\frac{v}{2})^2$ and $\frac{u-e}{2}$ and $\frac{u+e}{2d}$ are coprime positive integers. Therefore, there are some $a,b\in\mathbb{N}$ such that $a^2=\frac{u-e}{2}$ and $b^2=\frac{u+e}{2d}$. We infer that $|a^2-db^2|=1$ and $a<u$, which contradicts the fact that $\varepsilon$ is a fundamental unit of $\mathcal{O}_K$. Therefore, $2d\nmid u+1$ and $2d\nmid u-1$. Next let there be some $e\in\{-1,1\}$ such that $d\mid u+e$. Then $\frac{u-e}{4}\frac{u+e}{d}=(\frac{v}{2})^2$ and $\frac{u-e}{4}$ and $\frac{u+e}{d}$ are coprime positive integers. Consequently, there are some $g,h\in\mathbb{Z}$ such that $g^2=\frac{u-e}{4}$ and $h^2=\frac{u+e}{d}$. Observe that $|(2g)^2-dh^2|=2$, and thus there are some $a,b\in\mathbb{Z}$ such that $|a^2-db^2|=2$. Now let $d\nmid u+1$ and $d\nmid u-1$. Then either there is some $e\in\{-1,1\}$ such that $4q\mid u+e$ or there is some $e\in\{-1,1\}$ such that $4r\mid u+e$. First let there be some $e\in\{-1,1\}$ with $4q\mid u+e$. Then $\frac{u-e}{2r}\frac{u+e}{4q}=(\frac{v}{2})^2$ and $\frac{u-e}{2r}$ and $\frac{u+e}{4q}$ are coprime positive integers. We infer that there are some $a,b\in\mathbb{Z}$ with $a^2=\frac{u-e}{2r}$ and $b^2=\frac{u+e}{4q}$. Obviously, $|ra^2-2qb^2|=1$. If there is some $e\in\{-1,1\}$ such that $4r\mid u+e$, then it follows by analogy that $|qa^2-2rb^2|=1$ for some $a,b\in\mathbb{Z}$. This proves the claim.

\medskip\smallskip
Without restriction we can assume that $p=q$. Note that $|qa^2-2rb^2|\not=1$ for all $a,b\in\mathbb{Z}$ (since $|px^2-\frac{\mathsf{d}_K}{p}y^2|\not=4$ for all $x,y\in\mathbb{Z}$). It remains to show that $\pmb{\Big(}\frac{\alpha d/p}{p}\pmb{\Big)}=-1$ or $\pmb{\Big(}\frac{-\alpha p}{r}\pmb{\Big)}=-1$.

First let $m=1$ or $n=1$. Note that $\pmb{\Big(}\frac{r}{q}\pmb{\Big)}=\pmb{\Big(}\frac{q}{r}\pmb{\Big)}=-1$ by Lemma~\ref{lemma 4.3}. If $m=1$, then $\pmb{\Big(}\frac{\alpha d/p}{p}\pmb{\Big)}=\pmb{\Big(}\frac{2r}{q}\pmb{\Big)}=-1$. Moreover, if $n=1$, then $\pmb{\Big(}\frac{-\alpha p}{r}\pmb{\Big)}=\pmb{\Big(}\frac{q}{r}\pmb{\Big)}=-1$. Next let $m=5$ or $n=5$. If there are some $a,b\in\mathbb{Z}$ with $|a^2-db^2|=2$, then ($m=5$ and $\pmb{\Big(}\frac{2}{q}\pmb{\Big)}=1$) or ($n=5$ and $\pmb{\Big(}\frac{2}{r}\pmb{\Big)}=1$), a contradiction. It follows from the claim that there are some $a,b\in\mathbb{Z}$ such that $|ra^2-2qb^2|=1$. If $m=5$, then $\pmb{\Big(}\frac{r}{q}\pmb{\Big)}=1$, and thus $\pmb{\Big(}\frac{\alpha d/p}{p}\pmb{\Big)}=\pmb{\Big(}\frac{2r}{q}\pmb{\Big)}=-1$. Next let $n=5$. Then $\pmb{\Big(}\frac{2q}{r}\pmb{\Big)}=1$, and hence $\pmb{\Big(}\frac{q}{r}\pmb{\Big)}=-1$. This implies that $\pmb{\Big(}\frac{-\alpha p}{r}\pmb{\Big)}=\pmb{\Big(}\frac{q}{r}\pmb{\Big)}=-1$. From now on let $\{m,n\}\cap\{1,5\}=\emptyset$.

Assume that $m=7$. Then $n=3$ by Lemma~\ref{lemma 4.3}. Suppose that there are some $a,b\in\mathbb{Z}$ and some $\beta\in\{-1,1\}$ such that $a^2-db^2=2\beta$. Then $\pmb{\Big(}\frac{2\beta}{q}\pmb{\Big)}=\pmb{\Big(}\frac{2\beta}{r}\pmb{\Big)}=1$, and hence $\beta=\pmb{\Big(}\frac{2\beta}{q}\pmb{\Big)}=\pmb{\Big(}\frac{2\beta}{r}\pmb{\Big)}=-\beta$, a contradiction. It follows by the claim that there are some $a,b\in\mathbb{Z}$ and some $\beta\in\{-1,1\}$ such that $ra^2-2qb^2=\beta$. Observe that $\pmb{\Big(}\frac{r}{q}\pmb{\Big)}=\pmb{\Big(}\frac{\beta}{q}\pmb{\Big)}=\beta$ and $\pmb{\Big(}\frac{-2q}{r}\pmb{\Big)}=\pmb{\Big(}\frac{\beta}{r}\pmb{\Big)}=\beta$. This implies that $-\beta=-\pmb{\Big(}\frac{r}{q}\pmb{\Big)}=\pmb{\Big(}\frac{q}{r}\pmb{\Big)}=\pmb{\Big(}\frac{-2q}{r}\pmb{\Big)}=\beta$, a contradiction. Therefore, $m\not=7$. We infer that $m=3$ and $n\in\{3,7\}$. Note that $\pmb{\Big(}\frac{\alpha d/p}{p}\pmb{\Big)}=-\alpha\pmb{\Big(}\frac{r}{q}\pmb{\Big)}$ and $\pmb{\Big(}\frac{-\alpha p}{r}\pmb{\Big)}=-\alpha\pmb{\Big(}\frac{q}{r}\pmb{\Big)}=\alpha\pmb{\Big(}\frac{r}{q}\pmb{\Big)}$. In any case, we conclude that $\pmb{\Big(}\frac{\alpha d/p}{p}\pmb{\Big)}=-1$ or $\pmb{\Big(}\frac{-\alpha p}{r}\pmb{\Big)}=-1$.

\medskip\smallskip
CASE 5: $d=qrs$ for some odd distinct $q,r,s\in\mathbb{P}$ with $q\equiv 1\mod 4$ and $r\equiv s\mod 4$. Then $\mathsf{d}_K=qrs$ and $p\in\{q,r,s\}$. We infer by Lemma~\ref{lemma 3.3} that $r\equiv s\equiv 3\mod 4$. It follows from Lemma~\ref{lemma 4.3} that $\pmb{\Big(}\frac{q}{r}\pmb{\Big)}=-1$ or $\pmb{\Big(}\frac{q}{s}\pmb{\Big)}=-1$. Without restriction let $\pmb{\Big(}\frac{q}{r}\pmb{\Big)}=-1$. Then $\pmb{\Big(}\frac{r}{q}\pmb{\Big)}=-1$. Therefore, for each $a,b\in\mathbb{Z}$, $|ra^2-\frac{\mathsf{d}_K}{r}b^2|\not=4$. If $p=r$, then $\pmb{\Big(}\frac{-\alpha p}{q}\pmb{\Big)}=-1$.

Next let $p=q$. Since $|qa^2-\frac{\mathsf{d}_K}{q}b^2|\not=4$ for all $a,b\in\mathbb{Z}$, it follows from Lemma~\ref{lemma 3.3} that there are some $a,b\in\mathbb{Z}$ with $|sa^2-\frac{\mathsf{d}_K}{s}b^2|=4$. We have $\pmb{\Big(}\frac{q}{s}\pmb{\Big)}=\pmb{\Big(}\frac{s}{q}\pmb{\Big)}=1$. If $\alpha=-1$, then $\pmb{\Big(}\frac{-\alpha p}{r}\pmb{\Big)}=-1$. If $\alpha=1$, then $\pmb{\Big(}\frac{-\alpha p}{s}\pmb{\Big)}=-\pmb{\Big(}\frac{q}{s}\pmb{\Big)}=-1$.

Finally, let $p=s$. Clearly, $|sa^2-\frac{\mathsf{d}_K}{s}b^2|\not=4$ for all $a,b\in\mathbb{Z}$, and thus there are some $a,b\in\mathbb{Z}$ such that $|qa^2-\frac{\mathsf{d}_K}{q}b^2|=4$ by Lemma~\ref{lemma 3.3}. This implies that $\pmb{\Big(}\frac{rs}{q}\pmb{\Big)}=1$, and thus $\pmb{\Big(}\frac{s}{q}\pmb{\Big)}=\pmb{\Big(}\frac{r}{q}\pmb{\Big)}=-1$. Consequently, $\pmb{\Big(}\frac{-\alpha p}{q}\pmb{\Big)}=\pmb{\Big(}\frac{s}{q}\pmb{\Big)}=-1$.
\end{proof}

\begin{corollary}\label{corollary 4.5}
Let $t=|\{p\in\mathbb{P}:p\mid\mathsf{d}_K\}|$ and let $D_d\not=\emptyset$.
\begin{enumerate}
\item ${\rm N}(\varepsilon)=-1$ if and only if $t=2$. If these equivalent conditions are satisfied, then there are some distinct $p,q\in\mathbb{P}$ such that $p\equiv q\equiv 1\mod 4$ and $d\in\{2p,pq\}$.
\item ${\rm N}(\varepsilon)=1$ if and only if $t=3$. If these equivalent conditions are satisfied, then there are some odd distinct $p,q,r\in\mathbb{P}$ such that $(p\equiv 1\mod 4$, $q\equiv 3\mod 4$ and $d=pq)$ or $d=2pq$ or $(p\equiv 1\mod 4$, $q\equiv r\equiv 3\mod 4$ and $d=pqr)$.
\end{enumerate}
\end{corollary}

\begin{proof}
Since $|{\rm Pic}(\mathcal{O}_K)|=2$ by Theorem~\ref{theorem 4.4}, we have $t\in\{2,3\}$ by Lemmas~\ref{lemma 2.2} and~\ref{lemma 4.2}, and \cite[p. 100]{Br74}. It follows from Corollary~\ref{corollary 3.10}(1) that there is some ramified $\ell\in\mathbb{P}$ such that $|\ell a^2-\frac{\mathsf{d}_K}{\ell}b^2|\not=4$ for all $a,b\in\mathbb{Z}$.

\bigskip
(1) If ${\rm N}(\varepsilon)=-1$, then it follows from Lemma~\ref{lemma 2.2} that $t=2$. Now let $t=2$. It is a simple consequence of Lemma~\ref{lemma 4.2} that $p\not\equiv 3\mod 4$ for each $p\in\mathbb{P}$ with $p\mid d$. In particular, $d\not\equiv 3\mod 4$ and $d\not\in\mathbb{P}$. Consequently, there are some distinct $p,q\in\mathbb{P}$ such that $d=pq$. If $d\equiv 1\mod 4$, then $p\equiv q\equiv 1\mod 4$. Now let $d\equiv 2\mod 4$. Then $p=2$ or $q=2$. Without restriction let $q=2$. Then $d=2p$. Therefore, the additional statement is true. Assume that ${\rm N}(\varepsilon)=1$. It can be shown along the same lines as in Cases $1$ and $2$ of the proof of Theorem~\ref{theorem 4.4} (where $\ell$ plays the role of the prime ``$p$'' in the aforementioned proof) that we obtain a contradiction. This implies that ${\rm N}(\varepsilon)=-1$.

\bigskip
(2) Since $t\in\{2,3\}$, it is an immediate consequence of (1) that ${\rm N}(\varepsilon)=1$ if and only if $t=3$. Now let $t=3$ and let ${\rm N}(\varepsilon)=1$. If $d\equiv 1\mod 4$, then it follows from Lemma~\ref{lemma 3.3} that there are some odd distinct $p,q,r\in\mathbb{P}$ such that $p\equiv 1\mod 4$, $q\equiv r\equiv 3\mod 4$ and $d=pqr$. If $d\equiv 2\mod 4$, then there are some odd distinct $p,q\in\mathbb{P}$ such that $\mathsf{d}_K=8pq$, and hence $d=2pq$. If $d\equiv 3\mod 4$, then $\mathsf{d}_K=4d$, and thus there are some odd distinct $p,q\in\mathbb{P}$ such that $p\equiv 1\mod 4$, $q\equiv 3\mod 4$ and $d=pq$. This shows that the additional statement is true.
\end{proof}

\section{Classification and additional facts and remarks}\label{5}

In the final section, we present a theorem which classifies all the real quadratic number fields that possess an order $\mathcal{O}$ with $\min\Delta(\mathcal{O})>1$. More precisely, we classify the squarefree numbers $d$ according to the structure of $D_d$ (i.e., the set of unusual conductors of $d$). We start with three propositions that describe the structure of $D_d$ and $D_d^{\prime}$ in some special cases.

\begin{proposition}\label{proposition 5.1}
Let $\Omega=\{|D_d|:d\in\mathbb{N}_{\geq 2}, d$ is squarefree, $|{\rm Pic}(\mathcal{O}_K)|=2$ and ${\rm N}(\varepsilon)=1\}$. Then $\Omega=[0,7]\setminus\{5\}$.
\end{proposition}

\begin{proof}
It follows from \cite[pp. 22 and 23]{HK13a} and a computer verification that $|{\rm Pic}(\mathcal{O}_K)|=2$ and ${\rm N}(\varepsilon)=1$ for each $d\in\{15,30,34,51,165,429,1005\}$. It is now an easy consequence of Corollary~\ref{corollary 2.10}, Proposition~\ref{proposition 3.1}, Theorem~\ref{theorem 4.4} and a computer verification that $D_{15}=\{2,3,5,6,10,15,30\}$, $D_{30}=\{3\}$, $D_{34}=\emptyset$, $D_{51}=\{3,17,51\}$, $D_{165}=\{3,5,10,15\}$, $D_{429}=\{11,13,22,26,143,286\}$ and $D_{1005}=\{67,134\}$. We infer that $[0,7]\setminus\{5\}\subseteq\Omega$. It follows from Propositions~\ref{proposition 3.1} and~\ref{proposition 3.4} and Theorem~\ref{theorem 4.4} that $\Omega\subseteq [0,7]$. Let $x\in\Omega$ be such that $x\geq 5$. Then there is some $d\in\mathbb{N}_{\geq 2}$ such that $d$ is squarefree, ${\rm N}(\varepsilon)=1$ and $|D_d|=x$. Observe that $D_d$ contains at least two ramified primes by Proposition~\ref{proposition 3.1}. If $D_d$ contains at least three ramified primes, then $x=7$ by Corollary~\ref{corollary 2.10}. Now let $D_d$ contain precisely two ramified primes $p$ and $q$. Since $x\geq 5$, we infer by Corollary~\ref{corollary 2.10}, Proposition~\ref{proposition 3.1} and Theorem~\ref{theorem 4.4} that $2$ is inert and $\{p,q,pq\}\subseteq D_d\subseteq\{p,q,pq,2p,2q,2pq\}$. Since $x\geq 5$, this implies that $\{2p,2q\}\subseteq D_d$ or $2pq\in D_d$. It follows from Corollary~\ref{corollary 2.10} that $2pq\in D_d$. Therefore, $x=6$ by Theorem~\ref{theorem 2.6}(3).
\end{proof}

\begin{proposition}\label{proposition 5.2}
Let $\Omega^{\prime}=\{|D_d^{\prime}|:d\in\mathbb{N}_{\geq 2}, d$ is squarefree, $|{\rm Pic}(\mathcal{O}_K)|=2$ and ${\rm N}(\varepsilon)=-1\}$.
\begin{enumerate}
\item Let $d=2p$ for some $p\in\mathbb{P}$ with $p\equiv 1\mod 4$, let $|{\rm Pic}(\mathcal{O}_K)|=2$ and let ${\rm N}(\varepsilon)=-1$. Then $2\in D_d^{\prime}$ and $\{f\in D_d:f$ is even$\}\subseteq\{2,2p\}$.
\item Let $d\equiv 5\mod 8$, ${\rm N}(\varepsilon)=-1$ and $(\mathcal{O}_K^{\times}:\mathcal{O}_2^{\times})=3$. If $p\in D_d^{\prime}\cap\mathbb{P}$, then $2p\in D_d^{\prime}$.
\item $\Omega^{\prime}=[0,3]\cup\{6\}$.
\end{enumerate}
\end{proposition}

\begin{proof}
(1) First we show that $2\in D_d^{\prime}$. By Theorem~\ref{theorem 3.9}, it remains to show that $|{\rm Pic}(\mathcal{O}_2)|=|{\rm Pic}(\mathcal{O}_K)|$. There are some $u,v\in\mathbb{N}$ such that $\varepsilon=u+v\sqrt{d}$. Since $u^2-2pv^2=-1$, we find that $v$ is odd, and thus $\varepsilon\not\in\mathcal{O}_2$. Consequently, $(\mathcal{O}_K^{\times}:\mathcal{O}_2^{\times})=2$, and hence $|{\rm Pic}(\mathcal{O}_2)|=|{\rm Pic}(\mathcal{O}_K)|$.

\smallskip
Now we show that $\{f\in D_d:f$ is even$\}\subseteq\{2,2p\}$. Let $f\in D_d$ be even. Assume that $f\not\in\{2,2p\}$. It follows from Proposition~\ref{proposition 3.4} and \cite[Proposition 4.15]{Br-Ge-Re20} that there is some odd inert $q\in\mathbb{P}$ such that $f\in\{2q,2pq\}$. We infer by Theorem~\ref{theorem 2.6}(3) that $2q\in D_d$ and $\mathcal{O}_q$ is half-factorial. Therefore, $|{\rm Pic}(\mathcal{O}_{2q})|=|{\rm Pic}(\mathcal{O}_q)|=|{\rm Pic}(\mathcal{O}_K)|$, and hence $(\mathcal{O}_K^{\times}:\mathcal{O}_{2q}^{\times})=2(q+1)$ and $(\mathcal{O}_K^{\times}:\mathcal{O}_{q}^{\times})=q+1$. There are some $a,b\in\mathbb{N}$ such that $\varepsilon^{q+1}=a+b\sqrt{d}$. We have $b=\sum_{i=0,i\equiv 1\mod 2}^{q+1}\binom{q+1}{i}u^{q+1-i}v^i(2p)^{\frac{i-1}{2}}\equiv (q+1)u^qv\equiv 0\mod 2$. Since $\varepsilon^{q+1}\in\mathcal{O}_q$, it follows that $q\mid b$, and thus $\varepsilon^{q+1}\in\mathcal{O}_{2q}$, which contradicts $(\mathcal{O}_K^{\times}:\mathcal{O}_{2q}^{\times})=2(q+1)$.

\bigskip
(2) Let $p\in D_d^{\prime}\cap\mathbb{P}$. Observe that $2$ is inert and $p$ is odd and ramified. It follows from Corollary~\ref{corollary 4.5}(1) that $p\equiv 1\mod 4$. Moreover, $|{\rm Pic}(\mathcal{O}_p)|=|{\rm Pic}(\mathcal{O}_K)|=2$ by Theorem~\ref{theorem 3.9}, and thus $(\mathcal{O}_K^{\times}:\mathcal{O}_p^{\times})=p$. This implies that $p\mid (\mathcal{O}_K^{\times}:\mathcal{O}_{2p}^{\times})\mid 3p$. Assume that $(\mathcal{O}_K^{\times}:\mathcal{O}_{2p}^{\times})=p$. Then $\varepsilon^p\in\mathcal{O}_2$, and thus $3\mid p$ (since $(\mathcal{O}_K^{\times}:\mathcal{O}_2^{\times})=3$). Consequently, $p=3\not\equiv 1\mod 4$, a contradiction. We have $(\mathcal{O}_K^{\times}:\mathcal{O}_{2p}^{\times})=3p$, and hence $|{\rm Pic}(\mathcal{O}_{2p})|=|{\rm Pic}(\mathcal{O}_K)|$. Therefore, $2p\in D_d^{\prime}$ by Theorem~\ref{theorem 3.9}.

\bigskip
(3) Note that $|{\rm Pic}(\mathcal{O}_K)|=2$ and ${\rm N}(\varepsilon)=-1$ for each $d\in\{10,85,185,1285,5374184665\}$ by \cite[pp. 22 and 23]{HK13a} and by computer verification (see the tables below). We infer by Corollary~\ref{corollary 2.10}, Proposition~\ref{proposition 3.1}, Theorem~\ref{theorem 3.9} and a computer verification that $D^{\prime}_{10}=\{2,5,10\}$, $D^{\prime}_{85}=\{5,10,17,34,85,170\}$, $D^{\prime}_{185}=\{37\}$, $D^{\prime}_{1285}=\{257,514\}$ and $D^{\prime}_{5374184665}=\emptyset$. This implies that $[0,3]\cup\{6\}\subseteq\Omega^{\prime}$. Let $x\in\Omega^{\prime}$. Without restriction let $x>0$. Then there is some squarefree $d\in\mathbb{N}_{\geq 2}$ such that $|{\rm Pic}(\mathcal{O}_K)|=2$, ${\rm N}(\varepsilon)=-1$ and $x=|D_d^{\prime}|$. It follows from Corollary~\ref{corollary 4.5}(1) that there are some distinct $p,q\in\mathbb{P}$ with $\{r\in\mathbb{P}:r\mid\mathsf{d}_K\}=\{p,q\}$. If $2$ is not inert, then $D_d^{\prime}\subseteq\{p,q,pq\}$ by Proposition~\ref{proposition 3.1}, and thus $x\leq 3$. Now let $2$ be inert. Then $p$ and $q$ are odd, $d\equiv 5\mod 8$ and we infer by Proposition~\ref{proposition 3.1} that $D_d^{\prime}\subseteq\{p,q,pq,2p,2q,2pq\}$. Therefore, $x\leq 6$. Now let $x\geq 4$. Observe that $\{p,q,pq\}\subseteq D_d^{\prime}$ by Theorem~\ref{theorem 2.6}(3) and Corollary~\ref{corollary 2.10}, and hence $\{2p,2q,2pq\}\cap D_d^{\prime}\not=\emptyset$. If $2pq\in D_d^{\prime}$, then $x=6$ by Theorem~\ref{theorem 2.6}(3). Therefore, we can assume without restriction that $2p\in D_d^{\prime}$. We see that $\mathcal{O}_2$ is half-factorial by Theorem~\ref{theorem 2.6}(3), and thus $|{\rm Pic}(\mathcal{O}_2)|=|{\rm Pic}(\mathcal{O}_K)|$ and $(\mathcal{O}_K^{\times}:\mathcal{O}_2^{\times})=3$. It follows from (2) that $2q\in D_d^{\prime}$. Finally, we obtain $2pq\in D_d^{\prime}$ by Corollary~\ref{corollary 2.10}, and hence $x=6$.
\end{proof}

\begin{proposition}\label{proposition 5.3}
Let $d\in\mathbb{N}_{\geq 2}$ be squarefree.
\begin{enumerate}
\item Let $d=2pq$ for some odd distinct $p,q\in\mathbb{P}$. Then each $f\in D_d$ is odd.
\item Let $d=pq$ for some $p,q\in\mathbb{P}$ such that $p\equiv 1\mod 4$ and $q\equiv 3\mod 4$ and let $u,v\in\mathbb{N}$ be such that $\varepsilon=u+v\sqrt{d}$. If $v$ is even, then $D_d=\emptyset$.
\item Let $d=pqr$ for some distinct $p,q,r\in\mathbb{P}$ such that $q<r$, $p\equiv 1\mod 4$, $q\equiv r\equiv 3\mod 4$ and $d\equiv 5\mod 8$. If $3<|D_d|<6$, then $q=3$ and either $D_d=\{p,q,2p,pq\}$ or $D_d=\{q,r,2r,qr\}$.
\item If $f\in D_d$ with $|\{r\in\mathbb{P}:r\mid f,r$ is ramified$\}|=3$, then there are some $p,q\in\mathbb{P}$ such that $p\equiv 1\mod 4$, $q\equiv 3\mod 4$, $d=pq$ and $f=2pq$.
\end{enumerate}
\end{proposition}

\begin{proof}
(1) Let $f\in D_d$. Assume that $f$ is even. Since $2$ is ramified, it follows from Theorem~\ref{theorem 2.6}(3) that $2\in D_d$. We infer that $(\mathcal{O}_K^{\times}:\mathcal{O}_2^{\times})=2$. By Corollary~\ref{corollary 4.5}(2), we have ${\rm N}(\varepsilon)=1$. There are some $u,v\in\mathbb{N}$ such that $\varepsilon=u+v\sqrt{d}$, and hence $u^2-dv^2=1$. Clearly, $u$ is odd. Moreover, $1\equiv u^2-2pqv^2\equiv 1-dv^2\mod 8$, and thus $v$ is even. This implies that $\varepsilon\in\mathcal{O}_2$, a contradiction.

\bigskip
(2) Let $v$ be even. Since $(\mathcal{O}_K^{\times}:\mathcal{O}_2^{\times})=1$, we have that $2\not\in D_d$. Moreover, it can be shown along the same lines as in Case 3 of the proof of Theorem~\ref{theorem 4.4} that $p,q\not\in D_d$. Therefore, $D_d=\emptyset$ by Corollary~\ref{corollary 3.10}(1).

\bigskip
(3) Let $3<|D_d|<6$. It follows from Corollaries~\ref{corollary 2.10} and~\ref{corollary 4.5} that $D_d$ contains at most two ramified primes and that ${\rm N}(\varepsilon)=1$. If $D_d$ contains at most one ramified prime $k$, then $D_d\subseteq\{k,2k\}$ by Proposition~\ref{proposition 3.1} and Theorem~\ref{theorem 4.4}, which contradicts $|D_d|>3$. Consequently, $D_d$ contains precisely two distinct ramified primes $k$ and $\ell$. We infer by Corollary~\ref{corollary 2.10}, Proposition~\ref{proposition 3.1} and Theorem~\ref{theorem 4.4} that $\{k,\ell,k\ell\}\subseteq D_d\subseteq\{k,\ell,2k,2\ell,k\ell,2k\ell\}$ and that $\{2k,2\ell\}\cap D_d$ is a singleton. Without restriction let $2k\in D_d$ and $2\ell\not\in D_d$. Then $D_d=\{k,\ell,2k,k\ell\}$. It is sufficient to show that $\ell=3$. It follows from Theorems~\ref{theorem 2.6}(3) and~\ref{theorem 4.4} that $|{\rm Pic}(\mathcal{O}_{2\ell})|\not=|{\rm Pic}(\mathcal{O}_{\ell})|=|{\rm Pic}(\mathcal{O}_2)|=|{\rm Pic}(\mathcal{O}_K)|=2$. Since $2$ is inert, we find that $(\mathcal{O}_K^{\times}:\mathcal{O}_2^{\times})=3$, $\ell=(\mathcal{O}_K^{\times}:\mathcal{O}_{\ell}^{\times})\mid (\mathcal{O}_K^{\times}:\mathcal{O}_{2\ell}^{\times})\mid 3\ell$ and $(\mathcal{O}_K^{\times}:\mathcal{O}_{2\ell}^{\times})\not=3\ell$. Therefore, $(\mathcal{O}_K^{\times}:\mathcal{O}_{2\ell}^{\times})=\ell$, and thus $\varepsilon^{\ell}\in\mathcal{O}_2$. Since $(\mathcal{O}_K^{\times}:\mathcal{O}_2^{\times})=3$, we have $3\mid\ell$, and hence $\ell=3$.

\bigskip
(4) Let $f\in D_d$ be such that $|\{r\in\mathbb{P}:r\mid f,r$ is ramified$\}|=3$. Observe that $|\{r\in\mathbb{P}:r\mid\mathsf{d}_K\}|=3$. We infer by Lemma~\ref{lemma 3.3}, Corollary~\ref{corollary 4.5}(2) and the claim in the proof of Theorem~\ref{theorem 4.4} that there are some $p,q\in\mathbb{P}$ such that $p\equiv 1\mod 4$, $q\equiv 3\mod 4$ and $d=pq$. It follows from Proposition~\ref{proposition 3.4} that $f=2pq$.
\end{proof}

For each $a,b,c\in\mathbb{N}$ and $S\subseteq\mathbb{N}$ set $F_a=\{a\}$, $F_{a^{\prime}}=\{a,2a\}$, $F_{a,b}=\{a,b,ab\}$, $F_{a^{\prime},b}=\{a,b,2a,ab\}$, $F_{a,b^{\prime}}=\{a,b,2b,ab\}$, $F_{a^{\prime},b^{\prime}}=\{a,b,2a,2b,ab,2ab\}$, $F_{a,b,c}=\{a,b,c,ab,ac,bc,abc\}$ and $aS=\{as:s\in S\}$.

\begin{theorem}\label{theorem 5.4}
Let $D_d\not=\emptyset$ and let $\mathcal{T}=\{\ell\in\mathbb{P}:\ell\equiv 3\mod 4,\ell$ is inert$\,\}$. Then precisely one of the following conditions is satisfied:
\begin{enumerate}
\item $d=2p$ for some $p\in\mathbb{P}$ with $p\equiv 1\mod 4$ and $D_d\in\{F_2,F_{2,p}\cup pS:S\subseteq\mathcal{T}\}$.
\item $d=pq$ for some odd $p,q\in\mathbb{P}$ such that $p<q$, $p\equiv q\equiv 1\mod 4$, $d\equiv 1\mod 8$ and $D_d\in\{F_p\cup pS,F_q\cup qT,F_{p,q}\cup pS\cup qT\cup pq(S\cap T):S,T\subseteq\mathcal{T}\}$.
\item $d=pq$ for some odd $p,q\in\mathbb{P}$ such that $p<q$, $p\equiv 1\mod 4$, $q\equiv 1\mod 4$, $d\equiv 5\mod 8$ and $D_d\in\{F_p\cup pS,F_q\cup qT,F_{p^{\prime}}\cup pS\cup 2pU,F_{q^{\prime}}\cup qT\cup 2qV,F_{p,q}\cup pS\cup qT\cup pq(S\cap T),F_{p^{\prime},q^{\prime}}\cup pS\cup qT\cup pq(S\cap T)\cup 2pU\cup 2qV\cup 2pq(U\cap V):U\subseteq S\subseteq\mathcal{T},V\subseteq T\subseteq\mathcal{T}\}$.
\item $d=pq$ for some odd distinct $p,q\in\mathbb{P}$ such that $p\equiv 1\mod 4$, $q\equiv 3\mod 4$ and $D_d\in\{F_2,F_p,F_q,F_{2,p},F_{2,q},F_{p,q},F_{2,p,q}\}$.
\item $d=2pq$ for some odd $p,q\in\mathbb{P}$ with $p<q$ and $D_d\in\{F_p,F_q,F_{p,q}\}$.
\item $d=pqr$ for some distinct $p,q,r\in\mathbb{P}$ such that $q<r$, $p\equiv 1\mod 4$, $q\equiv r\equiv 3\mod 4$, $d\equiv 1\mod 8$ and $D_d\in\{F_p,F_q,F_r,F_{p,q},F_{p,r},F_{q,r}\}$.
\item $d=pqr$ for some distinct $p,q,r\in\mathbb{P}$ such that $q<r$, $p\equiv 1\mod 4$, $q\equiv r\equiv 3\mod 4$, $d\equiv 5\mod 8$ and $D_d\in\{F_p,F_q,F_r,F_{p^{\prime}},F_{q^{\prime}},F_{r^{\prime}},F_{p,q},F_{p^{\prime},q},F_{p^{\prime},q^{\prime}},F_{p,r},F_{p^{\prime},r^{\prime}},F_{q,r},F_{q,r^{\prime}},F_{q^{\prime},r^{\prime}}\}$.
\end{enumerate}
\end{theorem}

\begin{proof}
It is clear that at most one of the conditions is satisfied. By Corollary~\ref{corollary 4.5}, it is sufficient to consider the following cases. In this proof we use \cite[Proposition 4.15]{Br-Ge-Re20}, Theorem~\ref{theorem 2.6}, Corollary~\ref{corollary 2.10}, Propositions~\ref{proposition 3.1} and~\ref{proposition 3.4}, Corollary~\ref{corollary 3.10} and Theorem~\ref{theorem 4.4} without further mention.

\medskip
CASE 1: $d=2p$ for some $p\in\mathbb{P}$ with $p\equiv 1\mod 4$. Note that ${\rm N}(\varepsilon)=-1$ by Corollary~\ref{corollary 4.5}(1) and $2\in D_d$ by Proposition~\ref{proposition 5.2}(1). If $p\not\in D_d$, then $D_d=F_2$ by Proposition~\ref{proposition 5.2}(1). Now let $p\in D_d$ and set $S=\{s\in\mathbb{P}:s$ is odd and inert and $ps\in D_d\}$. Then $S\subseteq\mathcal{T}$ by Lemma~\ref{lemma 2.3} and $D_d=F_{2,p}\cup pS$ by Proposition~\ref{proposition 5.2}(1).

\medskip
CASE 2: $d=pq$ for some odd $p,q\in\mathbb{P}$ such that $p<q$, $p\equiv q\equiv 1\mod 4$ and $d\equiv 1\mod 8$. By Corollary~\ref{corollary 4.5}(1), we have ${\rm N}(\varepsilon)=-1$. Set $S=\{s\in\mathbb{P}:s$ is odd and inert and $ps\in D_d\}$, $T=\{s\in\mathbb{P}:s$ is odd and inert and $qs\in D_d\}$. Then $S,T\subseteq\mathcal{T}$ by Lemma~\ref{lemma 2.3}. Moreover, $\{s\in\mathbb{P}:s$ is odd and inert and $pqs\in D_d\}=S\cap T$. If $p\in D_d$ and $q\not\in D_d$, then $D_d=F_p\cup pS$. If $p\not\in D_d$ and $q\in D_d$, then $D_d=F_q\cup qT$. Finally, if $p,q\in D_d$, then $D_d=F_{p,q}\cup pS\cup qT\cup pq(S\cap T)$.

\medskip
CASE 3: $d=pq$ for some odd $p,q\in\mathbb{P}$ such that $p<q$, $p\equiv q\equiv1\mod 4$ and $d\equiv 5\mod 8$. Observe that ${\rm N}(\varepsilon)=-1$ by Corollary~\ref{corollary 4.5}(1). Set $S=\{s\in\mathbb{P}:s$ is odd and inert and $ps\in D_d\}$, $T=\{s\in\mathbb{P}:s$ is odd and inert and $qs\in D_d\}$, $U=\{s\in\mathbb{P}:s$ is odd and inert and $2ps\in D_d\}$, $V=\{s\in\mathbb{P}:s$ is odd and inert and $2qs\in D_d\}$. Then $U\subseteq S\subseteq\mathcal{T}$ and $V\subseteq T\subseteq\mathcal{T}$ by Lemma~\ref{lemma 2.3}. Furthermore, $\{s\in\mathbb{P}:s$ is odd and inert and $pqs\in D_d\}=S\cap T$ and $\{s\in\mathbb{P}:s$ is odd and inert and $2pqs\in D_d\}=U\cap V$. If $p\in D_d$, $2p\not\in D_d$ and $q\not\in D_d$, then $D_d=F_p\cup pS$. If $p\not\in D_d$, $q\in D_d$ and $2q\not\in D_d$, then $D_d=F_q\cup qT$. If $p,q\in D_d$ and $2p,2q\not\in D_d$, then $D_d=F_{p,q}\cup pS\cup qT\cup pq(S\cap T)$. If $2p\in D_d$ and $q\not\in D_d$, then $D_d=F_{p^{\prime}}\cup pS\cup 2pU$. If $p\not\in D_d$ and $2q\in D_d$, then $D_d=F_{q^{\prime}}\cup qT\cup 2qV$. By Proposition~\ref{proposition 5.2}(2), it remains to consider the case $2p,2q\in D_d$. We infer that $D_d=F_{p^{\prime},q^{\prime}}\cup pS\cup qT\cup pq(S\cap T)\cup 2pU\cup 2qV\cup 2pq(U\cap V)$.

\medskip
CASE 4: $d=pq$ for some odd distinct $p,q\in\mathbb{P}$ such that $p\equiv 1\mod 4$ and $q\equiv 3\mod 4$. If $\ell$ is the only ramified prime that is contained in $D_d$, then $D_d=F_{\ell}$. If $k$ and $\ell$ are precisely the two distinct ramified primes that are contained in $D_d$, then $D_d=F_{k,\ell}$. Finally, if $2,p,q\in D_d$, then $D_d=F_{2,p,q}$.

\medskip
CASE 5: $d=2pq$ for some odd $p,q\in\mathbb{P}$ with $p<q$. By Proposition~\ref{proposition 5.3}(1), it follows that each $f\in D_d$ is odd. If $p\in D_d$ and $q\not\in D_d$, then $D_d=F_p$. If $p\not\in D_d$ and $q\in D_d$, then $D_d=F_q$. Finally, if $p,q\in D_d$, then $D_d=F_{p,q}$.

\medskip
CASE 6: $d=pqr$ for some distinct $p,q,r\in\mathbb{P}$ such that $q<r$, $p\equiv 1\mod 4$, $q\equiv r\equiv 3\mod 4$ and $d\equiv 1\mod 8$. Since $2$ is neither inert nor ramified, we find that each $f\in D_d$ is odd. By Proposition~\ref{proposition 5.3}(4), $D_d$ cannot contain more than two ramified primes. If $\ell$ is the only ramified prime with $\ell\in D_d$, then $D_d=F_{\ell}$. If $k$ and $\ell$ are the two distinct ramified primes that are contained in $D_d$, then $D_d=F_{k,\ell}$.

\medskip
CASE 7: $d=pqr$ for some distinct $p,q,r\in\mathbb{P}$ such that $q<r$, $p\equiv 1\mod 4$, $q\equiv r\equiv 3\mod 4$ and $d\equiv 5\mod 8$. Proposition~\ref{proposition 5.3}(4) implies that $D_d$ contains at most two ramified primes. If $\ell$ is the only ramified prime that is contained in $D_d$ and $2\ell\not\in D_d$, then $D_d=F_{\ell}$. If $\ell$ is the only ramified prime that is contained in $D_d$ and $2\ell\in D_d$, then $D_d=F_{\ell^{\prime}}$. Now let $k$ and $\ell$ be the two distinct ramified primes that are contained in $D_d$. If $2k,2\ell\not\in D_d$, then $D_d=F_{k,\ell}$. If $2k,2\ell\in D_d$, then $D_d=F_{k^{\prime},\ell^{\prime}}$. Finally, we can assume without restriction that $2k\in D_d$ and $2\ell\not\in D_d$. Note that $3<|D_d|<6$, and hence $\ell=q=3$ and $D_d=F_{k^{\prime},q}=F_{q,k^{\prime}}$ by Proposition~\ref{proposition 5.3}(3).
\end{proof}

Let $D_d\not=\emptyset$ and set $m_1=2$, $m_2=m_5=3$, $m_3=m_6=6$, $m_4=7$ and $m_7=14$. For each $(A,B)\in\{(i,j):i\in [1,7],j\in [1,m_i]\}$, we say that $d$ is of type $A$/form $B$ if $d$ satisfies the $A$th condition of Theorem~\ref{theorem 5.4} and $D_d^{\prime}=F$ where $F$ is the first set (in the union) of the $B$th entry in the set of possible choices for $D_d$ in Theorem~\ref{theorem 5.4}. For instance, $d$ is of type $3$/form $2$ if and only if $d=pq$ and $D_d^{\prime}=F_q$ for some odd $p,q\in\mathbb{P}$ such that $p<q$, $p\equiv q\equiv 1\mod 4$ and $d\equiv 5\mod 8$. Furthermore, $d$ is of type $5$/form $3$ if and only if $d=2pq$ and $D_d^{\prime}=F_{p,q}$ for some odd $p,q\in\mathbb{P}$ with $p<q$. Next we want to point out that (with the possible exception of type $4$/form $1$) all types/forms actually occur. Note that the values of $d$ below are the smallest possible values for the given type and form.

\begin{table}[htbp]
\centering
\begin{tabular}{|c|c|c|c|c|c|c|c|c|c|c|c|}
\hline
$d$ & $1221562$ & $10$ & $302737$ & $185$ & $65$ & $4159957$ & $2581$ & $184861$ & $1285$ & $485$ & $85$\\
\hline
$D_d^{\prime}$ & $F_2$ & $F_{2,5}$ & $F_{97}$ & $F_{37}$ & $F_{5,13}$ & $F_{1777}$ & $F_{89}$ & $F_{401^{\prime}}$ & $F_{257^{\prime}}$ & $F_{5,97}$ & $F_{5^{\prime},17^{\prime}}$\\
\hline
$\textnormal{type}$ & $1$ & $1$ & $2$ & $2$ & $2$ & $3$ & $3$ & $3$ & $3$ & $3$ & $3$\\
\hline
$\textnormal{form}$ & $1$ & $2$ & $1$ & $2$ & $3$ & $1$ & $2$ & $3$ & $4$ & $5$ & $6$\\
\hline
\end{tabular}
\end{table}

\begin{table}[htbp]
\centering
\begin{tabular}{|c|c|c|c|c|c|c|c|c|c|c|c|c|c|c|c|}
\hline
$d$ & $267$ & $779$ & $87$ & $115$ & $51$ & $15$ & $30$ & $70$ & $66$ & $665$ & $1065$ & $1265$ & $385$ & $273$ & $105$\\
\hline
$D_d^{\prime}$ & $F_{89}$ & $F_{19}$ & $F_{2,29}$ & $F_{2,23}$ & $F_{17,3}$ & $F_{2,5,3}$ & $F_3$ & $F_7$ & $F_{3,11}$ & $F_5$ & $F_3$ & $F_{23}$ & $F_{5,7}$ & $F_{13,7}$ & $F_{3,7}$\\
\hline
$\textnormal{type}$ & $4$ & $4$ & $4$ & $4$ & $4$ & $4$ & $5$ & $5$ & $5$ & $6$ & $6$ & $6$ & $6$ & $6$ & $6$\\
\hline
$\textnormal{form}$ & $2$ & $3$ & $4$ & $5$ & $6$ & $7$ & $1$ & $2$ & $3$ & $1$ & $2$ & $3$ & $4$ & $5$ & $6$\\
\hline
\end{tabular}
\end{table}

\clearpage

\begin{table}[htbp]
\centering
\begin{tabular}{|c|c|c|c|c|c|c|c|c|c|c|c|c|c|c|}
\hline
$d$ & $9085$ & $861$ & $6357$ & $1965$ & $6461$ & $1005$ & $885$ & $165$ & $1085$ & $2301$ & $429$ & $1173$ & $357$ & $805$\\
\hline
$D_d^{\prime}$ & $F_5$ & $F_3$ & $F_{163}$ & $F_{5^{\prime}}$ & $F_{7^{\prime}}$ & $F_{67^{\prime}}$ & $F_{5,3}$ & $F_{5^{\prime},3}$ & $F_{5^{\prime},7^{\prime}}$ & $F_{13,59}$ & $F_{13^{\prime},11^{\prime}}$ & $F_{3,23}$ & $F_{3,7^{\prime}}$ & $F_{7^{\prime},23^{\prime}}$\\
\hline
$\textnormal{type}$ & $7$ & $7$ & $7$ & $7$ & $7$ & $7$ & $7$ & $7$ & $7$ & $7$ & $7$ & $7$ & $7$ & $7$\\
\hline
$\textnormal{form}$ & $1$ & $2$ & $3$ & $4$ & $5$ & $6$ & $7$ & $8$ & $9$ & $10$ & $11$ & $12$ & $13$ & $14$\\
\hline
\end{tabular}
\end{table}

Let $d\in\mathbb{N}_{\geq 2}$ be squarefree and let $u,v\in\mathbb{N}$ be such that $\varepsilon=u+v\omega$. Next we want to discuss the sequence of squarefree integers $d\in\mathbb{N}_{\geq 2}$ for which $d\mid v$. This sequence of integers was investigated in \cite{Ch-Sa14,Mo13,St-Wi88}. The first eight members (cf. OEIS A135735) are well known. The remaining members (that we were able to discover) may be known too, but we were unable to find them in the literature. We present the first $17$ members of the sequence in question below. Our main motivation to identify more members was twofold. On the one hand, we wanted to provide an example of a squarefree number $d\in\mathbb{N}_{\geq 2}$ such that $|{\rm Pic}(\mathcal{O}_K)|=2$, ${\rm N}(\varepsilon)=-1$ and $D_d=\emptyset$. (This is used in the proof of Proposition~\ref{proposition 5.2}(3).) On the other hand, we wanted to give an example of $d$ that is of type $4$/form $1$. (This is still an open problem.)

\smallskip
It is an easy consequence of Proposition~\ref{proposition 5.3}(2) and Theorems~\ref{theorem 4.4} and~\ref{theorem 5.4} (together with the arguments in Case 3 of the proof of Theorem~\ref{theorem 4.4}) that $d$ is of type $4$/form $1$ if and only if there are some $p,q\in\mathbb{P}$ such that $p\equiv 5\mod 8$, $q\equiv 3\mod 4$, $d=pq$, $|{\rm Pic}(\mathcal{O}_K)|=2$, $v$ is odd and $d\mid v$. Also note that if $d=209991$, $p=69997$ and $q=3$, then $p,q\in\mathbb{P}$, $p\equiv 5\mod 8$, $q\equiv 3\mod 4$, $d=pq$, $|{\rm Pic}(\mathcal{O}_K)|=2$, $v$ is even and $d\mid v$.

\begin{remark}\label{remark 5.5}
Let $d\in [2,1.5\cdot 10^{12}]$ be squarefree and let $u,v\in\mathbb{N}$ be such that $\varepsilon=u+v\omega$. Then $d\mid v$ if and only if $d\in\{46,430,1817,58254,209991,1752299,3124318,4099215,5374184665,6459560882\}\cup\{16466394154,20565608894,25666082990,117477414815,125854178626,1004569189366,1188580642033\}$.
\end{remark}

\begin{proof}
This can be proved by doing a computer search and thereby using the large step algorithm/small step algorithm of \cite{St-Wi88}.
\end{proof}

Our implementation of the large step algorithm/small step algorithm was written in C and the program was compiled with GCC-12.2.0. For the computer run we only utilized privately owned hardware. We used 140 CPU cores (with hyperthreading and a clock rate around 3.8 GHz on average) and spent approximately 650 hours for the search in total. We did an exhaustive search (for values of $d$) up to $10^{10}$ with the small step algorithm and (for values of $d$) up to $1.5\cdot 10^{12}$ with the large step algorithm. We verified each of the $17$ values for $d$ with both the small step algorithm and the large step algorithm as well as with Mathematica 12.0.0 and Pari/GP 2.15.2. For more information on the numbers involved, we refer to \cite{Mo13,St-Wi88}.

\smallskip
We continue with a few facts about the squarefree numbers in Remark~\ref{remark 5.5}. These facts were obtained by using both Mathematica 12.0.0 and Pari/GP 2.15.2. Let $\beta\in [0,7]$ be such that $d\equiv\beta\mod 8$ and let $t=|\{p\in\mathbb{P}:p\mid\mathsf{d}_K\}|$.

\begin{table}[htbp]
\centering
\begin{tabular}{|c|c|c|c|c|c|c|c|c|c|}
\hline
$d$ & $46$ & $430$ & $1817$ & $58254$ & $209991$ & $1752299$ & $3124318$ & $4099215$ & $5374184665$\\
\hline
$\beta$ & $6$ & $6$ & $1$ & $6$ & $7$ & $3$ & $6$ & $7$ & $1$\\
\hline
$t$ & $2$ & $3$ & $2$ & $5$ & $3$ & $4$ & $2$ & $4$ & $2$\\
\hline
${\rm N}(\varepsilon)$ & $1$ & $1$ & $1$ & $1$ & $1$ & $1$ & $1$ & $1$ & $-1$\\
\hline
$|{\rm Pic}(\mathcal{O}_K)|$ & $1$ & $2$ & $1$ & $8$ & $2$ & $4$ & $1$ & $4$ & $2$\\
\hline
\end{tabular}
\end{table}

\begin{table}[htbp]
\centering
\begin{tabular}{|c|c|c|c|c|c|}
\hline
$d$ & $6459560882$ & $16466394154$ & $20565608894$ & $25666082990$ & $117477414815$\\
\hline
$\beta$ & $2$ & $2$ & $6$ & $6$ & $7$\\
\hline
$t$ & $4$ & $4$ & $3$ & $4$ & $5$\\
\hline
${\rm N}(\varepsilon)$ & $1$ & $1$ & $1$ & $1$ & $1$\\
\hline
$|{\rm Pic}(\mathcal{O}_K)|$ & $4$ & $32$ & $2$ & $8$ & $8$\\
\hline
\end{tabular}
\end{table}

\begin{table}[htbp]
\centering
\begin{tabular}{|c|c|c|c|}
\hline
$d$ & $125854178626$ & $1004569189366$ & $1188580642033$\\
\hline
$\beta$ & $2$ & $6$ & $1$\\
\hline
$t$ & $4$ & $2$ & $3$\\
\hline
${\rm N}(\varepsilon)$ & $1$ & $1$ & $1$\\
\hline
$|{\rm Pic}(\mathcal{O}_K)|$ & $8$ & $1$ & $2$\\
\hline
\end{tabular}
\end{table}

For the sake of completeness, we want to point out that our implementation of the large step algorithm of \cite{St-Wi88} did find the number $d=17451248829$, while our implementation of the small step algorithm of \cite{St-Wi88} did not find this number. For a more elaborate discussion of this phenomenon (which can only occur if $d$ is divisible by $3$), we refer to \cite[p. 621]{St-Wi88}. Observe that if $d=17451248829$ and $u,v,u^{\prime},v^{\prime},u^{\prime\prime},v^{\prime\prime}\in\mathbb{N}$ are such that $\varepsilon=u+v\omega$ and $\varepsilon^3=u^{\prime}+v^{\prime}\omega=u^{\prime\prime}+v^{\prime\prime}\sqrt{d}$, then $d\nmid v$, $d\mid 3v$, $d\mid v^{\prime}$ and $d\mid v^{\prime\prime}$. Moreover, if $d=17451248829$, then $\beta=5$, $t=4$, ${\rm N}(\varepsilon)=1$ and $|{\rm Pic}(\mathcal{O}_K)|=4$.

\bigskip
Let $N_n=|\{\mathcal{O}_f:d\in\mathbb{N}_{\geq 2}, d$ is squarefree, $f\in D_d$ and $f^2\mathsf{d}_K\leq 10^n\}|$ for each $n\in\mathbb{N}$. In other words if $n\in\mathbb{N}$, then $N_n$ is the number of orders in real quadratic number fields for which the set of distances is unusual and whose discriminant is at most $10^n$.

\begin{remark}\label{remark 5.6}
The first few values of $N_n$ are given in the table below. Moreover, $160,240,416,540,560,$ $928,945$ and $1000$ are precisely the discriminants $\leq 1000$ of the orders $\mathcal{O}$ with $\min\Delta(\mathcal{O})>1$.
\begin{table}[htbp]
\centering
\begin{tabular}{|c|c|c|c|c|c|c|c|c|}
\hline
$n$ & $2$ & $3$ & $4$ & $5$ & $6$ & $7$ & $8$ & $9$\\
\hline
$N_n$ & $0$ & $8$ & $80$ & $583$ & $4455$ & $36308$ & $311437$ & $2741750$\\
\hline
\end{tabular}
\end{table}
\end{remark}

\begin{proof}
This can be proved by doing a computer search and with the aid of Theorems~\ref{theorem 3.9}, {}~\ref{theorem 4.4} and~\ref{theorem 5.4}.
\end{proof}

For the computer search we used both Mathematica 12.0.0 and Pari/GP 2.15.2. The computer hardware that was used for this specific search (an i5-2500) was provided by the Karl-Franzens-Universit\"at Graz.

\smallskip
Let $\mathcal{D}=\{\mathcal{O}_f:d,f\in\mathbb{N}, d>1, d$ is squarefree, $\min\Delta(\mathcal{O}_f)>1\}$ be the set of all orders in real quadratic number fields for which the set of distances is unusual.

\begin{proposition}\label{proposition 5.7}
Let $d\in\mathbb{N}_{\geq 2}$ be squarefree.
\begin{enumerate}
\item Let $d\in\{430,209991,5374184665,20565608894,1188580642033\}$. Then for each ramified $p\in\mathbb{P}$, we have $|{\rm Pic}(\mathcal{O}_p)|\not=|{\rm Pic}(\mathcal{O}_K)|=2$.
\item If $\{2^p-1:p\in\mathbb{P},p\equiv 3\mod 4\}\cap\mathbb{P}$ is infinite, then $\mathcal{D}$ is infinite.
\item If $\{p\in\mathbb{P}:p\equiv 1\mod 4,|{\rm Pic}(\mathcal{O}_{\mathbb{Q}(\sqrt{2p})})|=2,{\rm N}(\varepsilon_{\mathbb{Q}(\sqrt{2p})})=-1\}$ is infinite, then $\mathcal{D}$ is infinite.
\end{enumerate}
\end{proposition}

\begin{proof}
Part (1) is an immediate consequence of Remark~\ref{remark 5.5} (including the tables above) and the fact that the second component of the fundamental unit for $d=209991$ is even. Part (2) follows from Example~\ref{example 3.2}(1), and (3) is a simple consequence of Proposition~\ref{proposition 5.2}(1).
\end{proof}

\bigskip
\noindent {\bf ACKNOWLEDGEMENTS.} We want to thank A. Geroldinger and the referee for helpful comments and suggestions that improved the readability and quality of this paper.

\end{document}